\documentclass{article}
\usepackage[utf8]{inputenc}

\usepackage{geometry}
\usepackage{mathtools,amsfonts,amsthm,mathrsfs,amssymb,stmaryrd,bbm}
\usepackage[dvipsnames]{xcolor}
\usepackage{tikz}
\usepackage{standalone}
\usepackage[justification=centering]{caption}
\usepackage[colorlinks=true, citecolor=cyan]{hyperref}
\usepackage{comment}

\newtheorem{thm}{Theorem}
\newtheorem{lemma}[thm]{Lemma}

\newtheorem{prop}[thm]{Proposition}

\newtheorem{defin}{Definition}
\newtheorem{rem}{Remark}

\newcommand{\kT}{\mathfrak{T}}

\newcommand{\kB}{\mathfrak{B}}

\newcommand{\Tvec}{\mathcal{T}}
\newcommand{\Fvec}{\mathcal{F}}
\newcommand{\Lvec}{\mathcal{L}}
\newcommand{\Mvec}{\mathcal{M}}
\newcommand{\Rvec}{\mathcal{R}}
\newcommand{\parts}{\mathcal{P}}

\newcommand{\vsum}{\ominus}
\newcommand{\lsum}{\oslash}
\newcommand{\rsum}{\obslash}
\newcommand{\usum}{\odot}
\newcommand{\intval}[1]{\llbracket #1 \rrbracket}

\newcommand{\BigO}{\mathcal{O}}

\usepackage{geometry}
\usepackage{todonotes}

\newcommand{\PT}[1]{\todo[inline,backgroundcolor=red!20!white]{\scriptsize Paul: #1}}

\title{Stretched exponential asymptotics for bases of triangular bootstrap percolation }
\author{Andrew Elvey Price \and Juliette Schabanel \and Paul Th\'evenin }
\date{}

\begin{document}

\maketitle

\begin{abstract}
In this paper, we study a bootstrap percolation process on the finite triangular grid $\kT_n$ of side length $n$. We say that a subset $\eta$ of points in $\kT_n$ percolates if the final configuration, starting from $\eta$, is the whole grid $\kT_n$.
A basis of size $n$ is then a subset of points of $\kT_n$ of minimum cardinality which percolates. In this paper, we first prove that the generating function counting bases satisfies an algebraic differential equation. Then, by analysing a modified version of this equation, we prove that the number $t_n$ of bases of size $n$ exhibits a stretched exponential asymptotic behaviour. More precisely, we show that $t_n \sim c n!e^{\sqrt{12n}}n^{5/12}$, for some constant $c>0$. These bases were recently shown by the second author to be in bijection with $3$-permutations avoiding the patterns $(12, 12)$ and $(231, 312)$, so this represents to our knowledge the first proven example of an asymptotic stretched exponential appearing in the study of pattern avoiding permutations.
\end{abstract}

\section{Introduction}

Fix $n \geqslant 0$ and denote by $\mathfrak{T}_n$ the set of $n(n+1)/2$ points of the triangular grid of size $n$ (Figure~\ref{fig:intro}, left). We consider the following dynamical process on the set of configurations $\parts(\mathfrak{T}_n)$. Start from an initial configuration $\eta_0$. For each $s \geqslant 1$, define $\eta_{s+1}$ as follows: for any $x \in \mathfrak{T}_n$, $x\in\eta_{s+1}$ if and only if either $x \in \eta_s$, or there exists a size $2$ upward triangle  $\{x,y,z\}$ in $\mathfrak{T}_n$ such that $y, z \in \eta_s$ (see Figure~\ref{fig:filling process}). Define $\zeta\coloneqq\inf \{s \geqslant 0 \mid \eta_{s+1}=\eta_s\}$, the number of steps needed to reach the final configuration. We say that $\eta_0$ is a basis of $\kT_n$ if it contains exactly $n$ points and $\eta_{\zeta} = \kT_n$. See Figure~\ref{fig:intro} (right) for an example of a basis of size $5$.

\begin{figure}[ht]
    \centering
    \includegraphics[page=1, height=2.5cm]{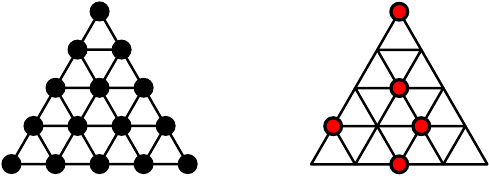}
    \caption{The set $\mathfrak{T}_5$, and a basis of size 5.}
    \label{fig:intro}
\end{figure}

These bases appear in \cite{SS23,SS25}, with connections to tilings and cellular automata theory. An analogous model on the infinite square lattice, known as Froböse percolation, has also been studied \cite{Fro89, HT24, SS25}.

Denote by $t_n$ the number of bases of $\kT_n$. Our main result is the asymptotic behaviour of the sequence $(t_n)_{n \geqslant 0}$.

\begin{thm}
\label{thm:asymptotic_behaviour}
There exists a constant $c>0$ such that, as $n \rightarrow \infty$:
\begin{align*}
t_n \sim c n! e^{\sqrt{12n}} n^{5/12}.
\end{align*}
\end{thm}
Our empirical analysis of the exact terms $t_{n}$ for $n\leqslant 500$ suggests that $c\approx0.00098107546$ (see Section \ref{subsec:triangle_empirical}).

\paragraph{Stretched exponential in asymptotics.}
The sequence $(t_n/n!)_{n \geqslant 0}$ is said to exhibit a \emph{stretched exponential} behaviour, in the sense that $\log \left( t_n/n! \right) \sim c n^q$, for some $c \neq 0$ and $0<q<1$. 
It is notoriously hard to rigorously prove that a sequence exhibits a stretched exponential behaviour, but there are some precedents. For instance, Revelle \cite{Rev03} showed that the cogrowth sequence of the lamplighter group behaves like $c \, 9^n \kappa^{n^{1/3}} n^{1/6}$, for some $c,\kappa >0$. Recently, the first author, with others, proved that the sequence enumerating so-called compacted binary trees, exhibits a stretched exponential behaviour with a factor $e^{c n^{1/3}}$ appearing in the enumeration \cite{EPFW21}. Their method was later adapted to show that families of automata \cite{EPFW20}, Young tableaux \cite{BW21} and phylogenetic networks \cite{FYZ21} also exhibit a stretched exponential factor in their enumeration.

Stretched exponentials have been conjectured to appear in cogrowth sequences of other groups \cite{EPG19}, and in other contexts, including pattern avoiding permutations \cite{CG15,CGZJ18}. 

\paragraph{Pattern avoiding 3-permutations.}

Bases of $\kT_n$ were recently shown by the second author \cite{BijPermut} to be in bijection with a family of pattern avoiding $3$-permutations, namely the class $Av((12,12), (231,312))$, following a conjecture of \cite{Bonichon_d-permutations}. Therefore our results apply equally to this class. 

A 3-permutation of length $n$ is a pair of permutations $(\sigma_{1},\sigma_{2})$ of $\intval{1, n}\coloneqq\{1,2,\ldots,n\}$. These can be naturally associated with the set of points $\{(i,\sigma_{1}(i),\sigma_{2}(i))\}_{i\in\intval{1,n}}$ in 3-dimensional space. A 3-permutation $(\sigma_{1},\sigma_{2})$ of length $n$ is said to contain a 3-permutation $(\pi_{1},\pi_{2})$ of length $m\leqslant n$ as a pattern if there are indices $1\leqslant i_{1}<\cdots<i_{m}\leqslant n$ such that $(\sigma_{1}(i_k))_{k\in\intval{1, m}}$ (resp. $(\sigma_{2}(i_k))_{k\in\intval{1, m}}$) has the same relative order as $(\pi_{1}(k))_{k\in\intval{1, m}}$ (resp. $(\pi_{2}(k))_{k\in\intval{1, m}}$). Otherwise $(\sigma_1, \sigma_2)$ {\em avoids} $(\pi_1, \pi_2)$. The set of 3-permutations avoiding $(\pi_{1},\pi_{2})$ is denoted $Av((\pi_{1},\pi_{2}))$, while $Av_{n}((\pi_{1},\pi_{2}))$ denotes those of length $n$.

For 3-permutations avoiding a given single pattern $(\pi_{1},\pi_{2})$ of length at least $2$, the number $|Av_{n}((\pi_{1},\pi_{2}))|$ is known to grow like $(n!)^{3/2}$ multiplied by a factor of at most exponential growth \cite[Theorem 3.4]{Gunby_d-permutations}. Obtaining exact results, however, or even more precise asymptotics for 3-permutations avoiding a single pattern, seems to be a very difficult problem. 

For pattern avoiding permutations avoiding multiple patterns, the first non-trivial classes are those avoiding one pattern of size 2 and one pattern of size 3. As pointed out in \cite[Proposition 6]{aldred2005permuting}, there are three such classes up to symmetries. The first class, $Av((12,12), (132, 312))$, was already enumerated exactly in \cite{ALW1995}:
\[|Av_{n}((12,12), (132, 312))|=(n+1)^{n-1}.\]
For the second class, $Av((12,12), (123, 321))$, the counting sequence remains unknown. The third class,  $Av((12,12), (231,312))$, is the one that we address in this article. In particular, Theorem \ref{thm:asymptotic_behaviour} can be restated as
\[|Av_{n}((12,12), (231,312))|\sim c n! e^{\sqrt{12n}} n^{5/12},\]
as $n \rightarrow \infty$, for some $c>0$.
To our knowledge, this is the first proven example of an asymptotic stretched exponential appearing in the study of pattern avoiding permutations. A stretched exponential is conjectured to appear in the asymptotics of other pattern avoiding classes such as $1324$-avoiding (classical) permutations \cite{CGZJ18}.

\hfill 

\paragraph{A D-algebraic equation.}
As a first step in the proof of Theorem \ref{thm:asymptotic_behaviour}, we derive an algebraic differential equation characterising the generating function, $T(z)$, of bases. We obtain this equation by applying an inclusion-exclusion argument to a decomposition from \cite{BijPermut}.
\begin{thm}
\label{thm:D-algebraic}
    Denoting $\Delta T(z) = (zT(z))'$, there is a series $I(z)$ such that $(T, I)$ is the unique pair of series satisfying the equations
    \begin{equation*}
        T(z) = 1 + I(z)\Delta T(z)
    \end{equation*}
    and 
    \begin{equation*}
        T(z) = 1 + z + 3I(z)(\Delta T(z)-1)-3I(z)^2\Delta^2T(z) + I(z)^3\Delta^3T(z).
    \end{equation*}
\end{thm}

The series $I(z)$ is actually the generating function of a certain subclass of bases which we call horizontally indivisible, see Section \ref{sec:indivisible bases}.

\paragraph{Differential equations and singularities.}
The question of recovering asymptotics of coefficients from a generating function that satisfies a differential equation has attracted a lot of attention, and we refer to \cite{FS09} for a history of the topic and main results. Intuitively, the nicer the equation is, the easier it is to extract information about the coefficients. In this direction, a well-studied class of functions is the class of \emph{D-finite} functions. A function $F$ is D-finite if it satisfies (on some complex domain $\Omega$) a linear differential equation of the form:
\begin{align*}
F^{(r)}(z) + a_1(z) F^{(r-1)}(z)+\ldots+a_r(z) F(z)=0,
\end{align*}
where $a_1,\ldots,a_r$ are (complex) rational functions. For $1 \leqslant j \leqslant r$ and $\zeta \in \Omega$, we define $\omega_\zeta(a_j)$ to be the order of the pole of $a_j$ at $\zeta$, with $\omega_\zeta(a_j)=0$ when $\zeta$ is not a pole of $a_j$. The differential equation is said to have a singularity at $\zeta$ if at least one of the $\left(\omega_\zeta(a_j)\right)_{1 \leqslant j \leqslant r}$ is positive. Furthermore, a singularity $\zeta$ is said to be \emph{regular} if $\omega_\zeta(a_j) \leqslant j$ for all $1 \leqslant j \leqslant r$, and irregular otherwise.

If a generating function $F$ is D-finite and its singularity with smallest modulus is unique and regular, then \cite[Theorem VII.10]{FS09} provides the asymptotic behaviour of the coefficients of $F$, which has a restricted form. In particular there can never be a stretched exponential in the asymptotics. For equations with an irregular dominant singularity, the asymptotic behaviour can be more diverse - and in particular include stretched exponentials - but there is no known automatic method to determine it. 

In our case, the function $T(z)$ is not D-finite, but only D-algebraic, meaning that it is a solution of an equation of the form $P(F(z),F^{(1)}(z), \ldots, F^{(k)}(z),z)=0$, where $P$ is a polynomial with complex coefficients and $k \geqslant 0$. Still, we manage to reduce the problem to the study of an (almost) linear differential equation with an irregular singularity, which explains the stretched exponential behaviour.

\paragraph{The method and its generality.}

Rather than analysing the equations of Theorem \ref{thm:D-algebraic} directly, we show in Section \ref{sec:subclasses} that the generating function $T(z)$ satisfies a linear differential equation (namely, \eqref{eq:quasi_D-finite}) with an irregular dominant singularity and an extra unknown term $F(z)$, whose coefficients are asymptotically negligible compared to those of $T(z)$ (Lemma \ref{lem_boundFT}).

Our main tool, Theorem \ref{thm:dichotomy_general}, states that, under certain technical conditions, coefficients of a generating function can only behave asymptotically in two different ways: either they grow at most polynomially, or they admit an explicit stretched exponential asymptotic form. Using the new equation \eqref{eq:quasi_D-finite}, we show in Section \ref{subsec:dichotomy} that this theorem applies in our case. Finally, in Section \ref{subsec:ruleout}, we rule out the possibility that the coefficients $t_n/n!$ grow at most polynomially, completing the proof of Theorem \ref{thm:asymptotic_behaviour}.

\smallskip

We expect our approach, and in particular Theorem \ref{thm:dichotomy_general}, to work in a general context, which would provide a generic method to determine asymptotics for D-finite series with stretched exponential behaviour. Additionally, our method applies to series that satisfy a linear differential equation with an extra unknown term, as long as we can sufficiently bound the extra term with respect to the original series.

\paragraph{Structure of the paper.}

In Section \ref{sec:defs}, we define the model and the main object of interest of the paper, that is, bases of triangular bootstrap percolation. Section \ref{sec:base decomposition} is devoted to the proof of Theorem \ref{thm:D-algebraic}, which we obtain through a canonical decomposition of bases. In Section \ref{sec:subclasses}, we investigate some subclasses of bases, which allows us to obtain a differential equation for $T(z)$, \eqref{eq:quasi_D-finite}, which is ``almost" linear. In Section \ref{sec:dichotomy result}, we state and prove a general technical result, Theorem \ref{thm:dichotomy_general}, for proving stretched exponential asymptotics of coefficients from generating functions. Finally, in Section \ref{sec:proof}, we complete the proof of Theorem \ref{thm:asymptotic_behaviour} by showing how \eqref{eq:quasi_D-finite} allows us to apply Theorem \ref{thm:dichotomy_general} in our case.

\section{Definitions}
\label{sec:defs}

For $n \geqslant 0$, denote by $\kT_n$ the set of points of the triangular grid of size $n$. A \emph{configuration} on $\kT_n$ is a subset $\eta \subseteq \kT_n$. We define the size of a configuration as its cardinality. 

The \emph{triangle bootstrap percolation process} is the dynamical process on the set of configurations defined as follows. Starting from an initial configuration $\eta$, one percolation step yields the configuration $\eta'$ such that for any $x \in \kT_n$, $x \in \eta'$ if and only if $x \in \eta$ or there is an upward triangle $\{x, y, z\}$ of size $2$ in $\kT_n$ such that $y, z \in \eta$. See Figure~\ref{fig:filling process} for an example of the process.

\begin{figure}[ht]
    \centering
    \includegraphics[page=2, width=\textwidth]{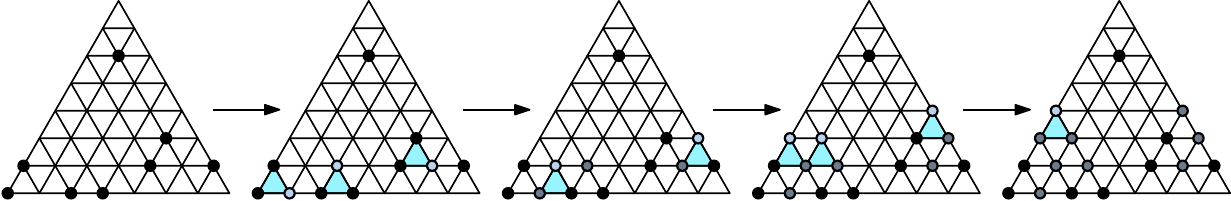}
    \caption{An example of the percolation process. The initial configuration is the black points and triangles involved in each step are highlighted in blue.}
    \label{fig:filling process}
\end{figure}

The triangle bootstrap percolation process is convergent and confluent \cite{SS23}. Convergence means that the process reaches a limit configuration from any initial configuration. Confluence means that adding the points one by one in any allowed order from a given initial configuration yields the same limit. We denote by $\varphi(\eta)$ the limit of the triangle bootstrap percolation process starting from the initial configuration $\eta$, and say that $\eta \subseteq \kT_n$ \emph{percolates} if $\varphi(\eta) = \kT_n$. A \emph{basis} is a percolating configuration of size $n$ (which is the minimal size of a percolating configuration, see \cite[Lemma 2]{SS23}). Figure~\ref{fig:basis ex} provides some examples. A \emph{subbasis} $\beta$ of a basis $\eta$ is a subset $\beta \subseteq \eta$ such that $\varphi(\beta)$ is a triangle of size $|\beta|$. For convenience, we      consider such a configuration as both a subset of $\eta$ and a basis of $\kT_{|\beta|}$.
\begin{figure}[ht]
    \centering
    \includegraphics[page=3, height=2.5cm]{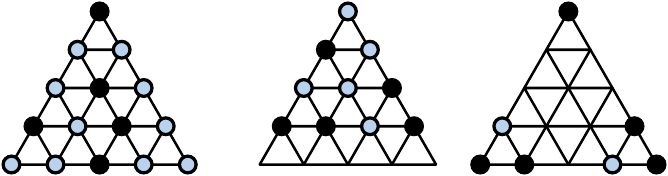}
    \caption{Left: A basis of $\kT_5$. Right: Two configurations that do not percolate. \\ The initial configurations are given by the black points, and the final configurations by the black and gray points.}
    \label{fig:basis ex}
\end{figure}

We denote by $\kB_n$ the set of bases of $\kT_n$. Observe that because of the symmetries of the process, $\kB_n$ is stable under rotation by $2\pi/3$ and symmetry in the vertical axis.

\section{Decomposition and equations for bases}
\label{sec:base decomposition}

In this section, we describe a recursive decomposition of bases and deduce equations on their numbers from it. This allows us in particular to derive Theorem \ref{thm:D-algebraic}. We also provide an empirical analysis of the asymptotics of the coefficients of $T(z)$, supporting the result of Theorem \ref{thm:asymptotic_behaviour}.

\subsection{Shifted sums}

Let us start by introducing the operation that will allow us to decompose bases.

\begin{defin}
    Let $n_1, n_2 >0$. Let  $\eta_1 \subseteq \kT_{n_1}$ and $\eta_2 \subseteq \kT_{n_2}$ be two configurations and $h \in \intval{0, n_1}$. The $h$-\emph{shifted horizontal sum} of $\eta_1$ and $\eta_2$ is the configuration $\eta_1 \vsum_h \eta_2 \subseteq \kT_{n_1+n_2}$ obtained by copying $\eta_1$ on the top $n_1$ rows and $\eta_2$ on the bottom $n_2$ rows of $\kT_{n_1+n_2}$ at the positions $\intval{h, h+n_2-1}$. An example is provided in Figure~\ref{fig:sum ex}.  Similarly, we define the $h$-\emph{shifted left sum} and $h$-\emph{shifted right sum}, denoted respectively $\eta_1 \lsum_h \eta_2$ and $\eta_1 \rsum_h \eta_2$, as the configurations obtained through a rotational symmetry of the construction for the $h$-shifted horizontal sum, i.e. by copying $\eta_1$ in one of the two bottom corners and $\eta_2$ on the remaining lines at a position determined by $h$ (see Figure~\ref{fig:sums scheme}). 
    
    In the case where $\beta_1$, $\beta_2$ are bases of respective sizes $n_1$ and $n_2$, we will then see $\beta_1$ and $\beta_2$ as subbases of $\beta_1 \vsum_h \beta_2$ in the natural way, so $\beta_1 \vsum_h \beta_2=\beta_1\cup\beta_2$. More generally, if $\beta_1$, $\beta_2$ and $\beta$ are subbases of a given basis, we will only say that $\beta=\beta_1\vsum_h\beta_2$, if this relation holds for the bases with the natural subbasis structure.

    We say that a configuration $\eta$ admits a horizontal (resp. left, resp. right) \emph{cut} into a pair $(\eta_1, \eta_2)$ if it is the horizontal (resp. left, resp. right) shifted sum of $\eta_1$ and $\eta_2$ with $\eta_1, \eta_2 \neq \varnothing$. We say that $\eta$ admits a cut into a pair $(\eta_1,\eta_2)$ if it admits a horizontal, left or right cut into $(\eta_1,\eta_2)$, that is, $\eta=\eta_1 \usum_h \eta_2$ for some $\usum \in \{ \vsum, \lsum, \rsum \}$ and some $h \geq 0$.
\end{defin}

\begin{figure}[ht]
    \centering
    \includegraphics[page=4, height=3.5cm]{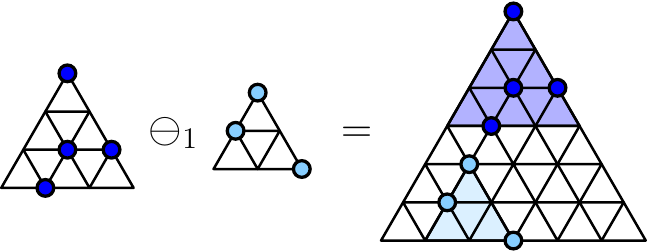}
    \caption{The $1$-shifted horizontal sum of two configurations.}
    \label{fig:sum ex}
\end{figure}

\begin{figure}[ht]
    \centering
    \includegraphics[page=5, height=3cm]{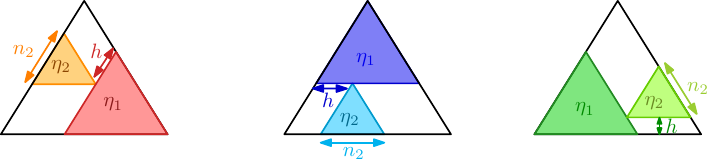}
    \caption{A schematic representation of the different $h$-shifted sums.\\ Left: $h$-shifted left sum. Middle: $h$-shifted horizontal sum. Right: $h$-shifted right sum.}
    \label{fig:sums scheme}
\end{figure}

\begin{thm}[{\cite[Proposition 3.8]{BijPermut}}]
\label{thm:sum}
    A configuration of size $n \geqslant 2$ is a basis if and only if it is the left, right or horizontal shifted sum of two bases.
\end{thm}

Observe that such a decomposition is not unique. Take for instance the configuration composed of the points on the bottom line of $\kT_n$: it admits a right and a left cut of any shift, so $2(n-1)$ cuts in total.

\subsection{Equations for bases}
\label{sec:indivisible bases}

For all $n \geqslant 0$, let $t_n$ denote the number of bases of size $n$. We define the generating function $T(z)$ of bases as 
\[T(z) = \sum_{n \geqslant 0} t_nz^n.\]
Its first terms are 
\[T(z) = 1 + z + 3z^2 + 16z^3 + 122z^4 + 1188z^5 + \BigO(z^6).\]
The goal of this section is to write the equations of Theorem~\ref{thm:D-algebraic} which characterise $T(z)$. To do this, we rely on the shifted sum decomposition from Theorem~\ref{thm:sum}.

\smallskip

First, let us define a canonical choice of decomposition for a basis. 

\begin{defin}
    A basis is \emph{horizontally indivisible} if it is not empty and does not admit a horizontal cut. Indivisibility in the other directions is defined similarly. We denote by $I(z)$ the generating function of horizontally indivisible bases. Observe that by symmetry, it is equal to the generating function of left (resp. right) indivisible bases. 
\end{defin}


Lemma~\ref{lem:uniqueness of the decomposition} will allow us to make a canonical choice of a cut in a given direction by requiring the second part to be indivisible. To prove it, we will need the following result.

\begin{lemma}
\label{lem:intersection}
    Let, $\beta$ be a basis and let $\beta_1$ and $\beta_2$ be two subbases of $\beta$. Then:
    \begin{itemize}
    \item[(i)] $\beta_1\cap\beta_2$ is also a subbasis of $\beta$.
    \item[(ii)] If $\beta = \beta_1 \usum_h \beta_2$ for some sum $\usum \in \{\vsum, \lsum, \rsum\}$ and shift $h \in \intval{0, |\beta_1|}$, then  any subbasis $\gamma$ of $\beta$ satisfies $\gamma = (\beta_1\cap\gamma) \usum_{h'} (\beta_2\cap\gamma)$ for some shift $h'$.
    \end{itemize}
\end{lemma}
\begin{proof}
    We say that $\omega := \{x_1,\ldots,x_{\frac{n(n-1)}{2}}\}$ is a percolation order of a basis $\beta$ of $\kT_n$ if $\omega \cup \beta = \kT_n$ and, for all $1 \leq i \leq n(n-1)/2$, there exists an upward triangle of size $2$ $\{x_i,y,z\}$ such that $y,z \in \beta \cup \{x_1,\ldots,x_{i-1}\}$. In other words, one can add the points of $\kT_n \backslash \beta$ in the order given by $\omega$ following the percolation rule.
    We start by proving that each point of $\kT_n \setminus\beta$ is added through the same triangle, independently of the percolation order.
    Let $T$ denote the set of upward triangles of size $2$ in $\kT_n$. Observe that any percolation order $\omega$ induces a map $f_\omega: T \to \kT_n \setminus \beta$, where $f_\omega(t)$ is the point of $t$ added in the percolation order. We want to prove that $f_\omega$ does not depend on $\omega$. First observe that since both sets have cardinality $n(n-1)/2$ and a given triangle cannot be used twice, $f_\omega$ is a bijection. 
    Then consider $\omega_0$ a percolation order, which exists since $\beta$ is a basis. Let $t_i$ be the triangle used at step $i$ in the order $\omega_{0}$, and $x_i$ the point added at that step. 
    We show the following by induction: for all $i \leqslant n$, $f_\omega(t_i)=x_i$ for all $\omega$.\\
    At step $1$, two points of $t_i$ are in the initial configuration, hence $t_1$ is necessarily used to add $x_1$ in any percolation order. So for all $\omega$, $f_{\omega}(t_1)=x_1$.\\
    Now assume that it holds up to step $i$ and consider the triangle $t_{i+1}$. Then, $t_{i+1}$ is made of $x_{i+1}$ and two points $a,b$ which either are in $\beta$, or were added before $x_{i+1}$ in $\omega_0$ and therefore have the same preimage by $f_\omega$ for all $\omega$, which is not $t_{i+1}$. Hence, $t_{i+1}$ can only be used to add $x_{i+1}$, independently of the percolation order. That is, $f_{\omega}(t_{i+1})=x_{i+1}$ for all $\omega$, which ends the induction. 
    
    Let us now prove (i). Denote by $k$ the size of the smallest triangle containing $\beta_1 \cap \beta_2$ and, by a small abuse of notation, denote by $\kT_k$ this triangle. First, since $\beta$ is a basis, we have according to \cite[Lemma 2]{SS23} that $|\beta_1 \cap \beta_2| \leqslant k$. Now consider a point $x$ in $\kT_k \setminus (\beta_1\cap\beta_2)$. Since both $\beta_1$ and $\beta_2$ are bases, $x$ was added in each of them through some upward triangles $T_1$ and $T_2$. Let $\omega_1, \omega_2$ be two percolation orders associated respectively to $\beta_1$ and $\beta_2$. Clearly one can complete both into percolation orders of $\beta$. By the result shown above, necessarily $x$ was added through the same triangle in both, so that $T_1=T_2\subseteq \kT_k$. Therefore, $x \in \varphi(\beta_1 \cap \beta_2)$. This shows that $\beta_1 \cap \beta_2$ percolates, so it is a basis. 
    
    In order to prove (ii), assume without loss of generality that $\beta = \beta_1 \vsum_{h} \beta_2$. Now observe that, under this assumption, $\beta_1\cap\gamma$ and $\beta_2\cap\gamma$ are disjoint subbases of $\beta$ whose union is $\gamma$, a subbasis, with $\beta_1\cap\gamma$ being made of the top rows of $\gamma$. This is exactly saying that $\gamma = (\beta_1\cap\gamma) \vsum_{h'} (\beta_2\cap\gamma)$ for some $h'$.
\end{proof}

We can now prove the key result to establish equations. 

\begin{lemma}
\label{lem:uniqueness of the decomposition}
    Let $\beta, \beta'$ be two bases and let $\eta, \eta'$ be two horizontally indivisible bases. If for some shifts $h, h'$ we have $\beta \vsum_h\eta = \beta'\vsum_{h'}\eta'$,  then $\beta = \beta'$, $\eta = \eta'$ and $h=h'$.
    
    The same result holds for the two other directions.
\end{lemma}
\begin{proof}
    Assume that $\beta \neq \beta'$. Then without loss of generality $\beta \subsetneq \beta'$ and so $\eta' \subsetneq \eta$. Let $\gamma = \beta' \cap \eta \neq \varnothing$. According to Lemma~\ref{lem:intersection}, $\gamma$ is a basis and $\eta = \gamma\vsum_s(\eta\cap\eta') = \gamma\vsum_s\eta'$ for some $s$. This is impossible as $\eta$ is horizontally indivisible, so $\beta = \beta'$.
    
    Now, as $\eta$ and $\eta'$ are composed of the points that do not belong to $\beta$, the only way to have $\eta \neq \eta'$ is to have $h\neq h'$. In this case, $\eta$ must have no points on one of the sides of $\kT_{|\eta|}$. This is impossible as, then, $\eta$ would not percolate. So $\eta = \eta'$ and $h=h'$.
\end{proof}

It follows from Lemma \ref{lem:uniqueness of the decomposition} that, when a basis admits several cuts in a given direction, there is only one cut such that the second part is indivisible in that same direction: the cut for which it has minimal size. This allows us to write the following more precise version of Theorem~\ref{thm:D-algebraic}, which we prove just after.

\begin{thm}
\label{thm:TandI}
    Denoting $\Delta T(z) \coloneqq (zT(z))'$, the generating functions $T(z)$ and $I(z)$ of bases and horizontally indivisible bases, respectively, are the unique series satisfying the equations
    \begin{equation}
    \label{eq:indiv eq}
        T(z) = 1 + I(z)\Delta T(z)
    \end{equation}
    and 
    \begin{equation}
    \label{eq:inclus-exclus}
        T(z) = 1 + z + 3I(z)(\Delta T(z)-1)-3I(z)^2\Delta^2T(z) + I(z)^3\Delta^3T(z).
    \end{equation}
\end{thm}
We start by establishing how to obtain the generating function of the sums of two families of bases.
\begin{lemma}
\label{lem:Delta}
    Let $\mathfrak{F}$ be a family of bases and let $\mathfrak{G}$ be a family of horizontally indivisible bases. Denote by $F(z)\coloneqq\sum f_nz^n$ and $G(z)\coloneqq\sum g_nz^n$ their respective generating functions. Then the generating function $(F\vsum G)(z)$ of the set $\mathfrak{F}\vsum\mathfrak{G}$ of horizontal sums of elements of $\mathfrak{F}$ and $\mathfrak{G}$ satisfies 
    \[(F\vsum G)(z) = \Delta F(z) \times G(z).\]
    
    The same equality holds for sums in another direction, provided that bases of $\mathfrak{G}$ are indivisible in that direction.
\end{lemma}
\begin{proof}
    To build a basis in $\mathfrak{F}\vsum\mathfrak{G}$, one should take a basis in $\mathfrak{F}$, one in $\mathfrak{G}$ and choose a shift between $0$ and the size of the first basis. This yields
    \[(F\vsum G)(z) \leqslant \sum_{n_1, n_2 \geqslant 0} (n_1+1)f_{n_1}z^{n_1}g_{n_2}z^{n_2} = \Delta F(z) \times G(z).\]
    
    Moreover, since all bases of $\mathfrak{G}$ are horizontally indivisible, according to Lemma~\ref{lem:uniqueness of the decomposition} two such sums never yield the same basis, so equality holds.
\end{proof}

\begin{proof}[Proof of Theorem~\ref{thm:TandI}]
    With Lemma~\ref{lem:Delta} in mind, \eqref{eq:indiv eq} is obtained by observing that a basis can either be empty, horizontally indivisible or admit a horizontal cut. In the last case, we choose the unique cut in which the second part $\eta_2$ is indivisible. This translates to the following equation of generating functions: 
    \[ T(z) = 1 + I(z) + I(z)(\Delta T(z) - 1) = 1+I(z)\Delta T(z).\]
    
    \smallskip
    
    The second equation \eqref{eq:inclus-exclus}, is obtained using the fact that any basis of size at least $2$ admits a cut (Theorem \ref{thm:sum}). There is only one basis of size $0$ and one of size $1$, so their contribution is $1+z$. Larger bases admit a cut and, by requiring that the second part of the cut is indivisible in the direction of the cut, we have a canonical choice for a cut in a given direction. This means that the number of bases admitting a cut in a given direction is $I(z)\Delta(T(z)-1)$, where the $-1$ corresponds to excluding the empty basis. Now, some bases may admit cuts in two different directions, or even in all three directions, so we need to use an inclusion-exclusion argument. 
    
    Let us count the number of bases admitting a cut in two different directions.  Let $\beta$ be a basis admitting a horizontal and a left cut and write $\beta = \beta_1 \vsum_h \beta_2 = \beta_1'\lsum_{h'} \beta_2'$ with $\beta_2, \beta_2'$ indivisible. We first prove that $\beta_2\cap\beta_2' = \varnothing$. 
    According to Lemma~\ref{lem:intersection}, both $(\beta_2\cap\beta_1')$ and $(\beta_2 \cap \beta_2')$ are subbases of $\beta_2$ and $\beta_2 = (\beta_2\cap\beta_1')\lsum_s\! (\beta_2 \cap \beta_2')$ for some $s \in \intval{0, h'}$. Similarly, we have $\beta_2' = (\beta_2'\cap\beta_1)\vsum_{s'}\! (\beta_2 \cap \beta_2')$. In fact, we must have either $s=0$ or $s' = |\beta_2'\cap\beta_1|$; otherwise $\varphi(\beta_1)$ and $\varphi(\beta_2)$ would intersect (see Figure~\ref{fig:proof eq2}), which is impossible since $(\beta_1, \beta_2)$ is a cut of the basis $\beta$.
    In particular one of the sums can be seen as either a horizontal or a left sum. Without loss of generality, assume that we can write $\beta_2' = (\beta_2\cap\beta_1')\lsum_0(\beta_2\cap\beta_2')= (\beta_2\cap\beta_1')\vsum_{|\beta_2\cap\beta_1'|}\!(\beta_2\cap\beta_2')$. Observe that $\beta_2\cap\beta_1'\neq\varnothing$: indeed, as $\beta_1'$ is a basis, it has a point on its bottom row, which is also the bottom row of $\beta$ (see Figure~\ref{fig:sums scheme}), so this point must belong to $\beta_2$. Therefore, since $\beta_2$ is horizontally indivisible, we must have $\beta_2\cap\beta_2'=\varnothing$. 
    
    Let us now go back to $\beta$. Any point of $\beta$ in the $|\beta_2|$ bottom lines of $\kT_{|\beta|}$ belongs to $\beta_2$, so $\beta_2'$ must lie above these lines, meaning that $h' \leqslant |\beta|-|\beta_2|$. The symmetric argument gives $h \geqslant |\beta_2'|$.  Hence, $h \in \intval{|\beta_2'|, |\beta_1|}$ and $h' \in \intval{0, |\beta_1'|-|\beta_2|}$.
    Then, according to Lemma~\ref{lem:intersection} (i), $\gamma \coloneqq \beta_1\cap\beta_1'$ is also a basis. Furthermore, since $\beta_2$ and $\beta_2'$ are disjoint, $\beta_1 = \gamma\lsum_{h'}\beta_2'$ and $\beta_1' = \gamma\vsum_{h-|\beta_2'|}\beta_2$, thus $\beta = (\gamma\vsum_{h-|\beta_2'|}\beta_2) \lsum_{h'}\beta_2' = (\gamma\lsum_{h'}\beta_2')\vsum_h\beta_2$. Moreover, any basis of the form $(\gamma_1\vsum_h\gamma_2)\lsum_{h'}\gamma_3$ admits both a horizontal and a left cut provided $0\leqslant h,h' \leqslant |\gamma_1|$. This means that bases admitting both an horizontal and a left cut are counted by $I(z)^2\Delta^2T(z)$ (here $\gamma = \varnothing$ is allowed). 
    \begin{figure}[ht]
    \centering
    \includegraphics[page=9, width=0.9\textwidth]{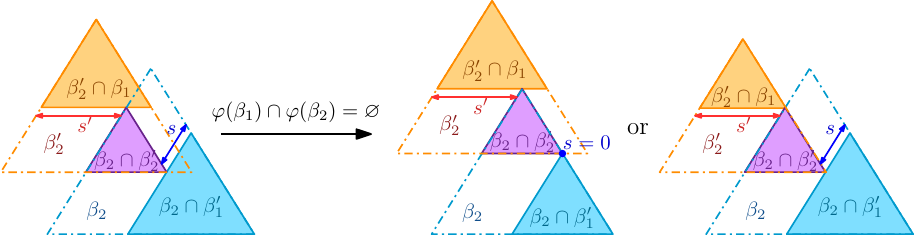}
    \caption{A schematic representation of the bases $\beta_2$ and $\beta_2'$. Left: The general position, which is incompatible with $\beta = \beta_1\vsum_h\beta_2$. Right: The two real possibilities: either $s=0$ or $s'= |\beta_2'\cap\beta_1|$.}
    \label{fig:proof eq2}
    \end{figure}
    
    Using the same arguments, a basis admitting a cut in all three directions can be written uniquely as $\beta = ((\beta_0\vsum_{h_1}\beta_1)\lsum_{h_2}\beta_2)\rsum_{h_3}\beta_3$ with $\beta_1$, $\beta_2$ and $\beta_3$ respectively horizontal, left- and right- indivisible bases and $h_1, h_2, h_3-|\beta_1| \in \intval{0, |\beta_0|}$. The conditions on the shifts implies that the shifted sums commute (with potentially different shifts). It follows that bases admitting cuts in all the three directions are counted by $I(z)^3\Delta^3T(z)$. 
    
    Finally, \eqref{eq:inclus-exclus} follows from the inclusion-exclusion argument illustrated in Figure~\ref{fig:inclus-exclus}.
    \begin{figure}[ht]
    \centering
    \includegraphics[page=6, width=0.8\textwidth]{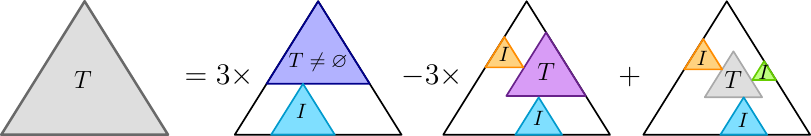}
    \caption{A schematic representation of the inclusion exclusion leading to \eqref{eq:inclus-exclus}. The $3\times$ correspond to rotational symmetries.}
    \label{fig:inclus-exclus}
    \end{figure}
\end{proof}

Equations \eqref{eq:indiv eq} and \eqref{eq:inclus-exclus} for $T(z)$ are non linear differential equations. Therefore, the function $T(z)$ is D-algebraic, a class of functions on which very little is known. Our strategy to analyse the asymptotic behaviour of the numbers $t_{n}$ will be to instead study some subclasses of bases with simpler, namely D-finite, generating functions. Defining these subclasses will be the object of Section~\ref{sec:subclasses}. Before that, we dedicate the next subsection to explaining how one can empirically predict the asymptotic behaviour of the number of bases from \eqref{eq:indiv eq} and \eqref{eq:inclus-exclus}.

\subsection{Empirical analysis of the asymptotics}
\label{subsec:triangle_empirical}

In this subsection we give empirical evidence that the number $t_{n}$ of triangle bases of size $n$ behaves like
\[t_{n}\sim c n! e^{\sqrt{12n}}n^{5/12},\]
with $c\approx 0.00098107545$. 

Using Theorem~\ref{thm:D-algebraic}, we computed the first 500 coefficients $t_{n}$ of the series $T(z)$ (see accompanying SageMath notebook). Given the following bounds from \cite{SS23}:
\[3n!\leqslant t_n \leqslant c\left(\frac{e}2\right)^nn^{n -\frac52} \text{ with } c >0,\]
we expect the sequence $(\frac{t_{n}}{nt_{n-1}})_{n \geqslant1}$ to converge to some constant. The most classical behaviour for such ratios would be
\[\frac{t_{n}}{nt_{n-1}}=c_{0}+c_{1}n^{-1}+\cdots+c_{k}n^{-k}+O(n^{-k-1}),\]
however this does not seem to be our case. Instead we assume the next simplest possibility
\[\frac{t_{n}}{nt_{n-1}}=c_{0}+c_{1}n^{-1/2}+\cdots+c_{k}n^{-k/2}+O(n^{-(k+1)/2}),\]
for any integer $k>0$. Then our empirical method works as follows: Fix $k$ to be some positive integer. Then, for each $n$, consider $m\in\intval{n,n+k-1}$. For each such $m$ we expect the equation
\[\frac{t_{m}}{mt_{m-1}}\approx\sum_{j=0}^{k-1}c_{j}m^{-j/2}\]
to be approximately true. We then solve this system of equations for $c_{0},\ldots,c_{k-1}$ as though the equations were exact, using known, exact values of $t_{m}$. This yields approximations for 
$c_{0},\ldots,c_{k-1}$. Denote the approximation thus obtained for $c_{j}$ by $v_{n,j}$. Note that this is equivalent to writing $v_{n,j}$ as a weighted sum of the numbers $\frac{t_{m}}{mt_{m-1}}$, which cancels the terms $n^{-s/2}$ for $s<k$ and $s\neq j$. Fixing $k=14$ then plotting the values $v_{n,0}$, $v_{n,1}$ and $v_{n,2}$ against $n$ convincingly yields $c_{0}=1$, $c_{1}=\sqrt{3}$ and $c_{2}=23/12$, suggesting that
\[\frac{t_{n}}{nt_{n-1}}=1+\frac{\sqrt{3}}{\sqrt{n}}+\frac{23}{12n}+O(n^{-3/2}),\]
which implies
\[t_{n}\sim c n! e^{\sqrt{12n}}n^{5/12}\]
for some constant $c>0$. Applying a similar analysis as above to the ratios $\frac{t_{n}}{n! e^{\sqrt{12n}}n^{5/12}}$, we predict that $c\approx 0.00098107546$, where we believe the digits quoted are correct.

\section{Quasi-linearisation of the differential equation via subclasses of bases}
\label{sec:subclasses}

In this section, we consider several subclasses of triangle bases with simpler, namely D-finite, generating functions and write equations for them. The main goal is to derive the following linear differential equation relating the generating function $T(z)$ to another series $F(z)$, and prove that $F(z)$ is negligible compared to $T(z)$.

\begin{lemma}
\label{lem:quasi_D-finite}
    There is a series $F(z)$ satisfying the coefficientwise inequality \[F(z) \leqslant cz^3T(z)\] for some constant $c>0$, such that 
    \begin{equation}
    \label{eq:quasi_D-finite} 
    T(z)=1+z+ (z+z^2+5z^3)^3\Delta^{3} T(z) -3(z+z^2+5z^3)^2\Delta^2 T(z)+ 3(z+z^2+5z^3)(\Delta T(z) -1) +F(z).
    \end{equation}
\end{lemma}

\subsection{Reducible bases}

Let us introduce a family of subclasses, $k$-reducible bases, which admit a simpler decomposition equation. 

\begin{defin}
For fixed $k>0$, a basis is \emph{$k$-reducible} if it is empty, has size at most $k$ or admits a cut into $(\beta_1, \beta_2)$ such that the second part, $\beta_2$, has size at most $k$, and the first part, $\beta_1$, is $k$-reducible (see Figure \ref{fig:irred ex}). In other words, a basis is $k$-reducible if it can be recursively decomposed as shifted sums of bases of sizes at most $k$. Observe that if a basis is $k$-reducible, then it is also $k'$-reducible for all $k'>k$.
\end{defin}

\begin{figure}[ht]
    \centering
    \includegraphics[height=2.5cm, page=7]{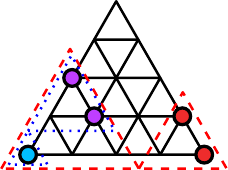}
    \caption{A basis that is $2$-reducible but also $1$-irreducible. The contours correspond to a possible decomposition into bases of sizes at most $2$.}
    \label{fig:irred ex}
\end{figure}

We denote by $R_k(z)$ the generating function of $k$-reducible bases. By mimicking the proof of Theorem \ref{thm:TandI}, we obtain the following characterisation of $R_k(z)$.
\begin{prop}
\label{prop:k-reducible_equation}
    Denote by $I_{k}(z)$ the polynomial counting horizontally indivisible bases of size at most $k$, that is $I_{k}(z)=[z^{\leqslant k}]I(z)$. The generating function $R_{k}(z)$ of $k$-reducible bases is the unique power series satisfying 
    \begin{equation}
    \label{eq:inclus-exclus_k}
        R_{k}(z) = 1 + z + 3I_{k}(z)(\Delta R_{k}(z)-1)-3I_{k}(z)^2\Delta^2R_{k}(z) + I_{k}(z)^3\Delta^3R_{k}(z).
    \end{equation}
\end{prop}
For example, the generating function $R_{1}(z)$ for 1-reducible bases satisfies
\[R_{1}(z)=1+z + 3z(\Delta R_{1}(z) -1)  -3z^2\Delta^2R_{1}(z) + z^3\Delta^{3} R_{1}(z).\]
Observe that while $T(z)$ is D-algebraic, $R_k(z)$ is D-finite, so it will be easier to study. Furthermore, since $k$-reducible bases form a subset of all bases, for all $k>0$ we have the coefficientwise inequality  $R_{k}(z)\leqslant T(z)$. 

\begin{rem}
Doing the same empirical analysis as in Section \ref{subsec:triangle_empirical}, it appears that the numbers of $1$-reducible bases of size $n$, $r^{(1)}_n$, satisfy 
\[\frac{n r^{(1)}_{n}}{t_{n}}\to c_1>0 \text{ with } c_1 \approx 7.0010417.\]
To improve this approximation, we consider 2-reducible bases. 
With still the same empirical analysis as before, it appears that the coefficients $r^{(2)}_n$ of $R_2(z)$ satisfy
\[\frac{r^{(2)}_{n}}{t_{n}}\to c_2>0 \text{ with } c_2\approx 0.568789102.\]
Increasing further the value of $k$ increases the constant toward $1$. 

A first strategy we considered was to prove that these ratios converge to a constant and analyse the asymptotic behaviour of $(r^{(2)}_n)_n$, which should be easier as $R_2(z)$ is D-finite. However, both steps turned out to be quite hard. In particular, the equation for $R_2(z)$ presents an irregular dominant singularity, which is at the origin of the stretched exponential asymptotic form but makes the automatic methods of analysis fail. 
\end{rem}

\subsection{Cores}

We now introduce the notion of \emph{$k$-core} of a basis, which can be seen as the part of the basis which is not $k$-reducible. This will allow us to write an equation similar to \eqref{eq:inclus-exclus_k} for $T(z)$.

\begin{defin}
For fixed $k>0$, a triangle basis $\beta$ is \emph{$k$-irreducible} if it has size greater than $k$ and does not admit any cut into a pair $(\beta_1, \beta_2)$, where the second part $\beta_2$ has size at most $k$ (see Figure~\ref{fig:irred ex}). 
\end{defin}


\begin{lemma}
\label{lem:core}
For each basis $\beta$, there is a unique well defined basis, called the $k$-core of $\beta$, with the following properties:
\begin{itemize}
    \item [(a)] If $\beta$ is $k$-irreducible then its $k$-core is $\beta$ itself.
    \item [(b)] If $\beta$ is $k$-reducible then its $k$-core is the empty basis.
   \item [(c)] If $\beta$ admits a cut into a pair $(\beta_1, \beta_2)$ with $\beta_2$ of size at most $k$, then the $k$-core of $\beta$ is the same as the $k$-core of $\beta_1$.
\end{itemize}
\end{lemma}
\begin{proof}
We will show this result by induction on the size $|\beta|$ of the basis $\beta$. For $|\beta|<k$, the basis $\beta$ is necessarily $k$-reducible, so its $k$-core is the empty basis. Note then that if $\beta$ admits a cut into a pair $(\beta_1,\beta_2)$ then (c) also holds as the $k$-core of $\beta_1$ is also the empty basis. 

For the inductive step, assume that $|\beta|=n$ and that the statement of the lemma holds for all bases of size at most $n-1$. If $|\beta|$ is $k$-irreducible, then its $k$-core is itself and the result holds. Otherwise, $\beta$ admits some cut into $(\beta_1, \beta_2)$, with $\beta_{2}$ of size at most $k$, and we can set the $k$-core of $\beta$ to be the same as the $k$-core of $\beta_{1}$. It remains to show that if $\beta$ admits any other cut into a pair $(\beta_1',\beta_2')$ with $|\beta_2'| \leq k$, the $k$-core of $\beta_{1}'$ is the same as the $k$-core of $\beta_1$. By Lemma~\ref{lem:intersection}, the pair $((\beta_{1}\cap\beta_{1}'), (\beta_{1}\cap\beta_{2}'))$ is a cut of $\beta_{1}$ for some $s \geq 0$.
 Then, by the induction hypothesis, $\beta_{1}$ satisfies (c), so the $k$-core of $\beta_1$ is the $k$-core of $\beta_1\cap\beta_1'$. Symmetrically the $k$-core of $\beta_1'$ is also the $k$-core of $\beta_1\cap\beta_1'$. This completes the induction and hence the proof of the lemma.  
\end{proof}

The $k$-core of a basis can be seen as the $k$-irreducible basis that remains after having recursively cut as many parts of size at most $k$ as possible. Observe that a $k$-core has size either $0$ or strictly greater than $k$, and that a basis has a $k$-core of size $0$ if and only if it is $k$-reducible.
\begin{figure}[ht]
    \centering
    \includegraphics[page=8, height=3.5cm]{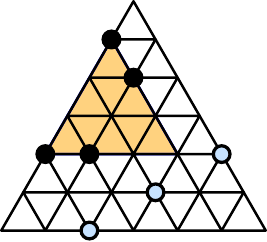}
    \caption{A basis with a $1$-core of size $4$. The black points are the $1$-core, and the gray ones the ``$1$-reducible part''.}
    \label{fig:core}
\end{figure}

For $k > 0$, let $F_{k}(z)$ denote the generating function for $k$-irreducible bases and let $f_{k, m}\coloneqq[z^{m}]F_{k}(z).$ Let also $T_{k, m}(z)$ be the generating function of bases with $k$-core of size $m$. Then we have the following equations:
\begin{equation}\label{eq:T_k_reducibility}
    T(z)=R_{k}(z)+\sum_{m=k+1}^{\infty}T_{k, m}(z),
\end{equation}
and for $k, m>0$, with the same inclusion-exclusion argument as for \eqref{eq:inclus-exclus} and \eqref{eq:inclus-exclus_k}, we have
\begin{equation}
    \label{eq:inclus-exclus_k_m}
        T_{k, m}(z) = f_{k, m}z^{m} + 3I_{k}(z)\Delta T_{k, m}(z)-3I_{k}(z)^2\Delta^2T_{k, m}(z) + I_{k}(z)^3\Delta^3T_{k, m}(z).
    \end{equation}
Combining \eqref{eq:inclus-exclus_k}, \eqref{eq:T_k_reducibility} and \eqref{eq:inclus-exclus_k_m} yields
\begin{equation}
    \label{eq:inclus-exclus_Fk}
        T(z) = 1+z+F_{k}(z) + 3I_{k}(z)(\Delta T(z)-1)-3I_{k}(z)^2\Delta^2T(z) + I_{k}(z)^3\Delta^3T(z)
    \end{equation}
for any $k>0$.

We now work with \eqref{eq:inclus-exclus_Fk} for a small fixed value of $k$, still large enough so that the coefficients of $F_k(z)$ are significantly smaller than those of $T(z)$. It turns out that $k=3$ is the smallest value that works.

Consider $3$-irreducible bases and their generating function $F_3(z)$. To lighten the notation, we drop the index, i.e. we denote $F(z) \coloneqq F_3(z)$ and $f_n \coloneqq f_{3, n}$. By specializing \eqref{eq:inclus-exclus_Fk} to $k=3$ we obtain that $T$ and $F$ satisfy \eqref{eq:quasi_D-finite}. The following lemma completes the proof of Lemma~\ref{lem:quasi_D-finite}.

\begin{lemma}
\label{lem_boundFT} 
    There is some constant $c>0$, such that for all $n \geqslant 3$, the number of triangle bases of size $n-3$ and the number of $3$-irreducible bases of size $n$ are related by the inequality
    \[f_{n}\leqslant c t_{n-3}.\] 
    In other words, we have the coefficientwise inequality 
    \[F(z)\leqslant c z^3 T(z).\]
\end{lemma}
\begin{proof}
Denote by $\mathfrak{F}_n$ the set of $3$-irreducible triangle bases of size $n$ (then $f_n = |\mathfrak{F}_n|$) and denote by $\mathfrak{H}_n$ the subset of $\mathfrak{F}_n$ of bases that admit a horizontal cut. 
Since we know that a triangle basis admits a cut in at least one of the three directions and $\mathfrak{F}$ is stable under rotation, we have 
\begin{equation}
\label{eq:f<=3h}
|\mathfrak{F}_{n}|\leqslant 3|\mathfrak{H}_{n}|.
\end{equation}
By definition, a given element $\beta$ of $\mathfrak{H}_{n}$ admits a horizontal cut, and the lower part has size at least $4$ (and can be chosen to be horizontally indivisible). In fact, the upper part of $\beta$ must also have size at least $3$, otherwise $\beta$ would admit a right or left cut $(\beta_1, \beta_2)$ with $\beta_2$ of size $1$ or $2$. Converting this to a counting argument with Lemma~\ref{lem:Delta} and using $I(z) \leqslant T(z)$ (coefficientwise), we obtain
\[|\mathfrak{H}_{n}|\leqslant \sum_{k=3}^{n-4}(k+1)t_{k}t_{n-k}.\]
We now proceed algebraically to bound the right-hand term. By separating the term for $k=3$ and grouping the terms for $k$ and $n-k$ for $k \geqslant 4$, we have
\[\sum_{k=3}^{n-4}(k+1)t_{k}t_{n-k}\leqslant 4t_{3}t_{n-3}+ \sum_{k=4}^{\lfloor\frac{n}{2}\rfloor}(n+2)t_{k}t_{n-k}.\]
Then, we bound $t_k$ and $t_{n-k}$ with two fairly weak inequalities: $t_{n}\geqslant nt_{n-1}$, obtained by counting bases with exactly one point on the bottom line, and $t_{n}\leqslant a\cdot b^{n} n!$, for some constant $a>0$ and $b=e^2/2<4$, from \cite[Theorem 5]{SS23}. This yields 
\[t_{n-k} \leqslant \frac{(n-k)!}{(n-3)!}t_{n-3} \text{ and } t_k \leqslant ab^kk!.\]
Hence, writing
\[d_{k}:=(n+2)ab^k k! \frac{(n-k)!}{(n-3)!}\geq (n+2)\frac{t_{k}t_{n-k}}{t_{n-3}},\]
we obtain
\[|\mathfrak{H}_{n}|\leqslant\sum_{k=3}^{n-4}(k+1)t_{k}t_{n-k}\leqslant 4t_{3}t_{n-3}+t_{n-3}\sum_{k=4}^{\lfloor\frac{n}{2}\rfloor}d_{k}.\]
Now, the ratio $d_{k+1}/d_k=b(k+1)/(n-k) \leqslant 1$ if and only if $k \leqslant (n-b)/(b+1)$. Hence, the sequence $(d_k)_{1 \leqslant k \leqslant \lfloor \frac{n}{2} \rfloor}$ decreases then increases, so we have $\max_{5 \leqslant k \leqslant \lfloor \frac{n}{2} \rfloor} d_k = \max(d_5,d_{\lfloor \frac{n}{2} \rfloor})$ (we will consider $d_4$ separately, as it is too large for this argument).

Let us now prove that, for $n$ large enough, $d_5 \geqslant d_{\lfloor \frac{n}{2} \rfloor}$. To estimate $d_{\lfloor \frac{n}{2} \rfloor}$, we use Stirling's formula. There is a constant $C>0$ such that, for all $n \geqslant 1$:
\[\frac{\lfloor\frac{n}{2}\rfloor!(n-\lfloor\frac{n}{2}\rfloor)!}{n!} \leqslant C\sqrt{\frac{n^2}{n}}\frac{(n/2)^{n/2}(n/2)^{n/2}}{n^n} = C\frac{\sqrt{n}}{4^{n/2}}.\]
Therefore, we have on the one hand
\[d_{\lfloor \frac{n}{2} \rfloor} := (n+2)a\cdot b^{\lfloor\frac{n}{2}\rfloor} \frac{\lfloor\frac{n}{2}\rfloor!(n-\lfloor\frac{n}{2}\rfloor)!}{(n-3)!} \leqslant aC(n-2)(n-1)n(n+2)\sqrt{n}\left(\!\frac{b}{4}\!\right)^{\frac{n}2}\]
which goes to $0$ exponentially fast as $n$ increases since $b<4$.
And on the other hand:
\begin{align*}
d_5 := (n+2) a b^5 5! \frac{(n-5)!}{(n-3)!} = a b^5 5! \frac{n+2}{(n-3)(n-4)} \geqslant d_{\lfloor \frac{n}{2} \rfloor}
\end{align*}
for $n$ large enough.
Therefore, for $n$ large enough we have 
\begin{align*}
|\mathfrak{H}_{n}| &\leqslant 4t_{3}t_{n-3}+\sum_{k=4}^{\lfloor\frac{n}{2}\rfloor}d_{k}t_{n-3} \\
&\leqslant 4t_{3}t_{n-3} +d_{4}t_{n-3} + \left(\left\lfloor\frac{n}{2}\right\rfloor-4\right)d_{5}t_{n-3},\\
&= 4t_{3}t_{n-3} + 4!ab^4\frac{(n+2)}{(n-3)}t_{n-3} + a(n+2)\left(\left\lfloor\frac{n}{2}\right\rfloor-4\right)\frac{5!b^5}{(n-3)(n-4)}t_{n-3} \leqslant ct_{n-3},
\end{align*}
for some $c>0$.
Using \eqref{eq:f<=3h}, we get
\[f_n \leqslant 3|\mathfrak{H}_n| \leqslant 3ct_{n-3}.\]


\end{proof}

\subsection{From D-finite series to P-recursive sequences}
\label{sec:recurrences}

It is known (see~\cite{Stanley2}) that one can derive a recurrence on the coefficients of a D-finite series. With this method, the differential equation~\eqref{eq:quasi_D-finite} can be rewritten as the following equations on the sequences $(t_n)_n$ and $(f_n)_n$:
\begin{equation}
\label{eq:t_and_f_recurrence}
t_{n+9}=m_{1}(n)t_{n+8}+m_{2}(n)t_{n+7}+m_{3}(n)t_{n+6}+\cdots+m_{8}(n)t_{n+1}+m_{9}(n)t_{n}+f_{n+9},
\end{equation}
where
\begin{align*}
    m_{1}(n)&=3 \, {\left(n + 9\right)}&
    m_{6}(n)&={\left(31 \, n + 49\right)} {\left(n + 4\right)}^{2}\\
    m_{2}(n)&=-3 \, {\left(n + 7\right)}{\left(n + 8\right)} &
    m_{7}(n)&=90 \, {\left(n + 3\right)}^{3}\\
    m_{3}(n)&={\left(n^{2} + 8 \, n + 22\right)} {\left(n + 7\right)}&
    m_{8}(n)&=75 \, {\left(n + 2\right)}^{3}\\
    m_{4}(n)&= 3 \,  {\left(n - 5\right)}{\left(n + 6\right)}^{2}&
    m_{9}(n)&=125\,  {\left(n+1\right)}^3.\\
    m_{5}(n)&= 6 \, {\left(3 \, n + 10\right)} {\left(n + 5\right)}^{2}&&
\end{align*}


Recall that, by considering the bases obtained as the horizontal shifted sum of a basis of size $n-1$ and the basis of size $1$, we have $t_n \geqslant n t_{n-1}$ for all $n \geqslant 1$, so $t_n \geqslant n!$. For this reason, it will be more comfortable to work with the sequence $\widehat{t}_{n}=\frac{1}{(n-1)!}t_{n}$. Furthermore, writing $\widehat{f}_{n}=\frac{1}{(n-1)!}f_{n}$ and $\widehat{m}_{j}(n)=\frac{(n+8-j)!}{(n+8)!}m_{j}(n)$ for $1 \leqslant j \leqslant 9$, \eqref{eq:t_and_f_recurrence} translates to the following equation on $\widehat{t}_n$:
\begin{equation}
\label{eq:recthat}
\widehat{t}_{n+9}=\widehat{m}_{1}(n)\widehat{t}_{n+8}+\widehat{m}_{2}(n)\widehat{t}_{n+7}+\widehat{m}_{3}(n)\widehat{t}_{n+6}+\cdots+\widehat{m}_{8}(n)\widehat{t}_{n+1}+\widehat{m}_{9}(n)\widehat{t}_{n}+\widehat{f}_{n+9},
\end{equation}
for $n \geqslant 1$. 

We prove in Section~\ref{sec:dichotomy result} a general dichotomy result concerning the asymptotic behaviour of sequences, and apply it in Section~\ref{sec:proof} to the sequence $(\widehat{t}_n)_n$.

\section{A general dichotomy result for asymptotics of sequences}
\label{sec:dichotomy result}

In this section, we state a general dichotomy result concerning the asymptotic growth of real sequences. Under some conditions - which are satisfied by the sequence $(\widehat{t}_n)_{n \geqslant 0}$ defined above, a real sequence either grows at most polynomially, or has an explicit stretched exponential asymptotic behaviour.

More precisely, we will control the sequence with an auxiliary sequence $(\alpha_n)_n$. Then, in the stretched exponential asymptotic case, we will deduce the exact asymptotic behaviour of our sequence of interest from the asymptotic behaviour of $(\alpha_n)_n$. The following lemma provides the tool for this conversion.

\begin{lemma}
\label{lem:log_deriv}
Consider sequences $\boldsymbol{x} := (x_n)_{n \geqslant 1}$ and $\boldsymbol{\alpha} := (\alpha_n)_{n \geqslant 1}$ of non-zero real numbers satisfying for all $n$:
\begin{align*}
\frac{x_{n+1}}{x_{n}}=\alpha_{n} = 1+\sum_{i=1}^p \frac{c_i}{n^{r_i}} + O(n^{-1-\eta}), 
\end{align*}
for some $\eta>0$, $p \geqslant 1$, $(c_i)_{i \in \llbracket 1,p \rrbracket} \in (\mathbb{R}^*)^p$ and $0<r_1<\ldots<r_p \leqslant 1$. Then the sequence $\boldsymbol{x}$ has the asymptotic form 
\[x_{n}\sim c e^{H(n)}n^{\gamma},\]
where $H(n)$ and $\gamma$ are defined as follows. 

We can write
\begin{equation}\label{eq:log_alpha}
\ln \alpha_n = \sum_{j=1}^q \frac{d_j}{n^{s_j}} +\frac{\gamma}{n}+ O(n^{-1-\eta'})    
\end{equation}
for some $\eta'>0$, $q \geqslant 1$, $(d_j)_{j \in \llbracket 1,q \rrbracket} \in (\mathbb{R}^*)^q$, $\gamma\in\mathbb{R}$ and $0<s_1<\ldots<s_q <1$. Then the constant $\gamma$ comes from \eqref{eq:log_alpha} and the function $H\in \mathbb{R}[X^{1-s_1},\ldots,X^{1-s_p}]$ is given by
\[ H(X)= \sum_{j=1}^{q} \frac{d_j}{1-s_j} X^{1-s_j}.\]
We call the pair $(H,\gamma)$ the log-integral of $\boldsymbol{\alpha}$.
\end{lemma}
\begin{proof}
First note that $\ln(\alpha_{n})$ can be written in the form \eqref{eq:log_alpha}: it suffices to substitute $\alpha_n$ by its expansion in the expansion of the logarithm
$\ln(\alpha_{n})=-\sum\limits_{k=1}^{\infty}\frac{(1-\alpha_{n})^{k}}{k}.$

Then, using the relation between $x_{n}$ and $\alpha_{n}$ and a telescoping product, we have
\[x_{n}=x_1\prod_{k=1}^{n-1}\alpha_k = x_{1}\exp\left(\sum_{k=1}^{n-1}\ln\alpha_{k}\right).\]
The sum on the right hand side can then be calculated at $n\to\infty$ using \eqref{eq:log_alpha}:
\[\sum_{k=1}^{n-1}\ln\alpha_{k}=\sum_{j=1}^q\left( \sum_{k=1}^{n-1}\frac{d_j}{k^{s_j}}\right) +\sum_{k=1}^{n-1}\frac{d}{n}+ O(n^{-\eta'})=\sum_{j=1}^q\frac{d_j}{1-s_j} n^{1-s_j} +d\ln(n)+K+ o(1)\]
for some $K\in\mathbb{R}$. That is:
\[x_{n}\sim x_{1}\exp\left(H(n)+\gamma\ln(n)+K\right).\]
This completes the proof of Lemma \ref{lem:log_deriv}, using $c=x_{1}e^{K}.$
\end{proof}
\begin{rem}
 As an example, if $\alpha_n=1+\frac{a}{\sqrt{n}}+\frac{b}{n}$, then $H(n)=2a\sqrt{n}$ and $\gamma=b-a^2/2$ so \[x_n \sim c \, e^{2a\sqrt{n}} n^{b-a^2/2}.\] 
 More generally, if $r_1,\ldots,r_p \in \mathbb{Q}$, then $H(n)=Q(n^r)$ where $r \in \mathbb{Q}$ is rational and $Q \in \mathbb{R}[X]$ is a polynomial.
\end{rem}

We now state our main technical tool, which can be seen as a generalisation of the lemma above.
\begin{thm}
\label{thm:dichotomy_general}
Fix reals $\varepsilon>0$ and $\Gamma>0$, and an integer $m \geqslant 1$. Consider $(m+1)$ sequences $(\alpha_n^{(j)})_{n \geqslant 1}$ for $j=0,1,\ldots,m$ satisfying:
\begin{itemize}
\item[(A1)] As $n \rightarrow \infty$, we have 
\begin{align*}
\alpha_n^{(0)} = 1 + \sum_{i=1}^p \frac{c_i}{n^{r_i}}+O(n^{-1-\eta}),
\end{align*}
for some $p \geqslant 1$, $0<r_1<...<r_p \leqslant 1$, $(c_i)_{i \in \llbracket 1,p\rrbracket} \in (\mathbb{R}^*)^p$ and $\eta>0$. We assume in addition that $c_1>0$ and $r_{1}<1$.
\item[(A2)] There exists $C>0$ such that, for all $1 \leqslant j \leqslant m$ and all $n \geqslant 1$ 
\begin{align*}
0 \leqslant \alpha_n^{(j)} \leqslant 1+\frac{C}{n}.
\end{align*}
\end{itemize}
Then there exists a constant $N_{0}$ (depending on $\Gamma$, $\varepsilon$ and the sequences $(\alpha_n^{(j)})_{n \geqslant 1}$) with the following property: For any $N>N_{0}$, any $(\lambda_n^{(1)},\ldots,\lambda_n^{(m)})_{n \geq N}$ real numbers such that $1+ \sum_{j=1}^m |\lambda_n^{(j)}| \leqslant \Gamma$ for all $n \geq N$, and any sequences of real numbers $x_{N}^{(j)},x_{N+1}^{(j)},\ldots$ and $y_{N},y_{N+1},\ldots$ satisfying:
\begin{align}
    y_{n}&=x_n^{(0)} + \sum_{j=1}^{m}\lambda_n^{(j)} x_{n}^{(j)}\label{eq:y_definition_general}\\
     n^{-1-2\varepsilon}y_{n}&>|x_{n+1}^{(j)}-\alpha_{n}^{(j)}x_{n}^{(j)}|,\label{eq:yx_inequality0_general}
\end{align}
for $n\geqslant N$ and $j=0,1,\ldots,m$, we have the following dichotomy:
\begin{itemize}
    \item[(a)] If there is some $M\geqslant N$ for which $x_{M}^{(0)}\geqslant \max\limits_{j \in \llbracket 1,m\rrbracket}|x_{M}^{(j)}|$ then it holds for all $n \geqslant M$, and, as $n\to\infty$, $y_{n}\sim x_{n}^{(0)} \sim c e^{H(n)} n^\gamma$ for some constant $c>0$, where $(H,\gamma)$ is the log-integral of $(\alpha_n^{(0)})_{n \geqslant 1}$.
    \item[(b)] If for all $n\geqslant N$ we have $x_{n}^{(0)}<\max\limits_{j \in \llbracket 1,m\rrbracket}|x_{n}^{(j)}|$, then $y_{n}$ grows at most polynomially.
\end{itemize}
\end{thm}

\begin{proof}
Consider $C>0$ so that (A2) holds with that constant $C$. Using (A1) and the fact that $r_1<1$ and $c_1>0$, we have $n (\alpha_n^{(0)}-1) \rightarrow \infty$ as $n \rightarrow \infty$. Therefore, there is some $N_0$ sufficiently large such that for any $n \geqslant N_0$:
\begin{align}
\label{eq:alpha0 is large}
\alpha_n^{(0)} \geqslant \left( 1 + \frac{C}{n} \right)\left( 1+\frac{1}{n} \right)^{\varepsilon} + \Gamma n^{-1-2\varepsilon} \left( 1+(n+1)^{\varepsilon} \right),
\end{align}
since the right-hand side is equivalent, as $n \rightarrow \infty$, to $1+\widetilde{C}/n$ for some constant $\widetilde{C} > 0$. 

For the first part of the proof we assume that there is some $M\geqslant N$ satisfying
\begin{equation}\label{eq:x_assumption_general}
x_{M}^{(0)}\geqslant \max\limits_{j \in \llbracket 1,m\rrbracket}|x_{M}^{(j)}|.
\end{equation}

We will show the following inequality by induction for $n\geqslant M$. \begin{equation}
\label{eq:x_inequality_general}
x_{n}^{(0)}\geqslant\left(\frac{n}{M}\right)^{\varepsilon}\max\limits_{j \in \llbracket 1,m\rrbracket}|x_{n}^{(j)}|.
\end{equation}
The base case is precisely the assumption \eqref{eq:x_assumption_general}. For the inductive step, we assume $x_{n}^{(0)}\geqslant\left(n/M\right)^\varepsilon|x_{n}^{(j)}|$ for all $j \geqslant 1$, and prove the same inequality for $n+1$. From the inequality \eqref{eq:yx_inequality0_general}, we have
\[x_{n+1}^{(0)}\geqslant \alpha_{n}^{(0)}x_{n}^{(0)}-n^{-1-2\varepsilon}|y_{n}|,\]
while for $j \geqslant 1$, we have
\[|x_{n+1}^{(j)}|\leqslant \alpha_{n}^{(j)}|x_{n}^{(j)}|+n^{-1-2\varepsilon}|y_{n}|\leqslant\alpha_{n}^{(j)}\left(M/n\right)^\varepsilon x_{n}^{(0)}+n^{-1-2\varepsilon}|y_{n}|.\]

By the inductive hypothesis, $x_{n}^{(0)}\geqslant|x_{n}^{(j)}|$ for all $j \geqslant 1$, so by \eqref{eq:y_definition_general} we have $|y_{n}|\leqslant \Gamma x_{n}^{(0)}$, since $\Gamma\geqslant 1+\sum_{j=1}^m |\lambda_n^{(j)}|$. Hence it suffices to prove that
\[\alpha_{n}^{(0)}x_{n}^{(0)}-\Gamma n^{-1-2\varepsilon} x_{n}^{(0)} \geqslant \left( \frac{n+1}{M} \right)^{\varepsilon} \left( \alpha_n^{(j)} (M/n)^\varepsilon x_n^{(0)} + \Gamma n^{-1-2\varepsilon} x_n^{(0)} \right), \] 
that is,
\[\alpha_{n}^{(0)}-\Gamma n^{-1-2\varepsilon}\geqslant \alpha_{n}^{(j)}\left(1+\frac{1}{n}\right)^\varepsilon+\Gamma\left(\frac{n+1}{M}\right)^\varepsilon n^{-1-2\varepsilon}.\]
From (A2), it is enough to prove that 
\[\alpha_{n}^{(0)}-\Gamma n^{-1-2\varepsilon}\geqslant \left( 1+\frac{C}{n}\right)\left(1+\frac{1}{n}\right)^\varepsilon+\Gamma \left(\frac{n+1}{M}\right)^\varepsilon n^{-1-2\varepsilon},\]
which follows directly from \eqref{eq:alpha0 is large}, completing the induction. 
Hence, \eqref{eq:x_inequality_general} holds for all $n\geqslant M$.

Remark now that, for all $n \geqslant M$, we have $x_n^{(0)}>0$. Indeed, if $x_n^{(0)}=0$, then $x_n^{(j)}=0$ for all $j \geqslant 1$ by \eqref{eq:x_inequality_general}, so that $y_n=0$ by definition (see \eqref{eq:y_definition_general}). This would contradict \eqref{eq:yx_inequality0_general}.

From \eqref{eq:y_definition_general} and \eqref{eq:x_inequality_general}, it follows that $y_{n}\sim x_{n}^{(0)}$ as $n\to\infty$. Hence, the inequality \eqref{eq:yx_inequality0_general} for $j=1$ can be written as
\[\frac{x_{n+1}^{(0)}}{x_{n}^{(0)}}=\alpha_{n}^{(0)}+O(n^{-1-2\varepsilon}).\]
Therefore, the log-integral of $(x_{n+1}^{(0)}/x_{n}^{(0)})_{n \geqslant 1}$ is the same as the log-integral $(H,\gamma)$ of $\boldsymbol{\alpha} := (\alpha_n^{(0)})_{n \geqslant 1}$. It follows from Lemma \ref{lem:log_deriv}, that 
\[ x_n^{(0)} \sim c \, e^{H(n)} n^\gamma,\]
for some $c \in \mathbb{R}$. Finally, since $y_{n}\sim x_{n}^{(0)}$, this completes the proof of (a).

\hfill

Let us now prove (b). We make the alternative hypothesis, that for all $n\geqslant N$ we have $x_{n}^{(0)}~<~\max\limits_{j \in \llbracket 1,m \rrbracket} |x_{n}^{(j)}|$.
 Set $z_{n}=\sum_{j = 1}^m|x_{n}^{(j)}|$, so that $x_{n}^{(0)}<z_{n}$. By \eqref{eq:y_definition_general}, we have
 \begin{equation}
 \label{eq:proof_of_b_general}
 y_{n}=x_{n}^{(0)}+\sum_{j=1}^m \lambda_n^{(j)} x_{n}^{(j)} <z_{n}+\sum_{j=1}^m |\lambda_n^{(j)}| |x_{n}^{(j)}|\leqslant \Gamma z_{n}. 
 \end{equation}

Moreover, from \eqref{eq:yx_inequality0_general}, we have for $n$ large enough
\[z_{n+1}\leqslant \sum_{j=1}^m |\alpha_{n}^{(j)}x_{n}^{(j)}|+m n^{-1-2\varepsilon}y_{n}\leqslant \left(\max_{j \in \llbracket 1,m\rrbracket}\alpha_{n}^{(j)}\right)z_{n}+m \Gamma n^{-1-2\varepsilon}z_{n},\]
where the second inequality uses \eqref{eq:proof_of_b_general}.

Dividing both sides by $z_{n}$ (which is positive) and using (A2) yields
\[\frac{z_{n+1}}{z_{n}}\leqslant 1+O\left(1/n\right).\]
It follows that $z_{n}$ grows at most polynomially. Finally, by \eqref{eq:yx_inequality0_general}, $y_{n}$ is positive so from \eqref{eq:proof_of_b_general}, $y_{n}$ grows at most polynomially.
\end{proof}

\section{Proof of Theorem \ref{thm:asymptotic_behaviour}}
\label{sec:proof}

The purpose of this section is to prove Theorem~\ref{thm:asymptotic_behaviour}, which states that the sequence $(t_n)_{n\geqslant0}$ has a stretched exponential asymptotic behaviour, by applying Theorem \ref{thm:dichotomy_general} to the sequence $(\widehat{t}_n)_{n \geqslant 0}$. To this end, we build suitable sequences $(x_n^{(j)})_{n \geqslant 1}$ and $(\alpha_n^{(j)})_{n \geqslant 1}$, then apply Theorem \ref{thm:dichotomy_general} to obtain a dichotomy result for $(\widehat{t}_n)_{n \geqslant 0}$. Finally, we show that this sequence grows faster than any polynomial, completing the proof of Theorem \ref{thm:asymptotic_behaviour}.

\subsection{Finding sequences \texorpdfstring{$(x_n^{(j)})_{n \geqslant 1}$}{x} and \texorpdfstring{$(\alpha_n^{(j)})_{n \geqslant 1}$}{alpha}}
\label{subsec:x_and_alpha}

In this section we provide the missing link between the recurrences \eqref{eq:t_and_f_recurrence} and Theorem \ref{thm:dichotomy_general}, namely we define from \eqref{eq:t_and_f_recurrence} sequences $(x_n^{(j)})_{n \geqslant }$, $1(\alpha_n^{(j)})_{n \geqslant 1}$ and $(\lambda_n^{(j)})_{n \geqslant 1}$, to which we can apply Theorem \ref{thm:dichotomy_general}. The method of proof is inspired by \cite{Wasow}, which uses a similar analysis on differential equations rather than recurrences. 


First, we define matrices associated to the recurrences \eqref{eq:t_and_f_recurrence} as

\begin{equation*}
    \Mvec_{n}:=\left(\begin{array}{cccccc}
0&1&0&\cdots&0&0\\
0&0&1&\cdots&0&0\\
\vdots&\vdots&\vdots&\ddots&\vdots&\vdots\\
0&0&0&\cdots&1&0\\
0&0&0&\cdots&0&1\\
\widehat{m}_{9}(n)&\widehat{m}_{8}(n)&\widehat{m}_{7}(n)&\cdots&\widehat{m}_{2}(n)&\widehat{m}_{1}(n)
\end{array}\right)\end{equation*}
and 
\begin{equation*}
\Tvec_{n}:=\left(\begin{array}{c}
\widehat{t}_{n}\\
\widehat{t}_{n+1}\\
\vdots\\
\widehat{t}_{n+6}\\
\widehat{t}_{n+7}\\
\widehat{t}_{n+8}
\end{array}
\right)\qquad\text{and}\qquad\Fvec_{n}:=\left(\begin{array}{c}
0\\
0\\
\vdots\\
0\\
0\\
\widehat{f}_{n+8}
\end{array}
\right).
\end{equation*}

Then \eqref{eq:recthat}  can be written as 
 \begin{equation}
 \label{eq:tlhat_matrix_recurrence}
 \Mvec_{n}\Tvec_{n}+\Fvec_{n+1}=\Tvec_{n+1}.
 \end{equation}
 
Our goal now is to apply Theorem \ref{thm:dichotomy_general} using sequences $(x_n^{(j)})_{n \geqslant 1}$ of the form $x_{n}^{(j)}=\Rvec^{(j)}_{n}\Tvec_{n}$, for suitably chosen row vectors $(\Rvec^{(j)}_{n})_{n \geqslant 1}$, as these are linear combinations of consecutive terms of the sequence $\widehat{t}_{n}$ of interest. We will then choose the sequences $(\alpha_n^{(j)})_{n \geqslant 1}$ accordingly. To apply Theorem~\ref{thm:dichotomy_general}, it is required that $x_{n+1}^{(j)} \approx \alpha_n^{(j)} x_n^{(j)}$. Using the fact that $\mathcal{F}_n$ is small, \eqref{eq:tlhat_matrix_recurrence} then implies that $\alpha_n^{(j)} \Rvec^{(j)}_n \Tvec_{n} \approx \Rvec^{(j)}_{n+1} \Tvec_{n+1} \approx \Rvec^{(j)}_{n+1} \Mvec_{n} \Tvec_{n}$. We therefore look for sequences that satisfy $\alpha_n^{(j)} \Rvec^{(j)}_n \approx \Rvec^{(j)}_{n+1} \Mvec_{n}$. 

Specifically, the level of precision that we will require is


\begin{equation}
\label{eq:R_generic_equation_first_time}
\left|\Rvec^{(j)}_{n+1}\Mvec_{n}- \alpha^{(j)}_{n} \Rvec^{(j)}_{n}\right|=O\left(n^{-\frac{5}{2}}\right),
\end{equation}
where $|\mathcal{X}|$ denotes the $1$-norm of a vector $\mathcal{X}$.

Since we expect the asymptotic behaviour of $\widehat{t}_n$ to contain a stretched exponential factor $e^{\sqrt{12 n}}$, Theorem \ref{thm:dichotomy_general} requires $\alpha_n^{(0)}$ to have the form $1+a/\sqrt{n}+b/n+O(n^{-1-\eta})$ for some $\eta>0$. Therefore, it makes sense from the relation $\alpha_n^{(j)} \Rvec^{(j)}_n \approx \Rvec^{(j)}_{n+1} \Mvec_{n}$ to look for vectors $\Rvec^{(j)}_n$ whose coefficients are polynomials in $1/\sqrt{n}$. Furthermore, by \eqref{eq:R_generic_equation_first_time} we can assume that these polynomials have degree at most $4$.

After some experimentation, it suffices to look for row vectors of the following generic form:
\[\Rvec^{*}_{n}:=\left(\begin{array}{c}
0\\
0\\
0\\
0\\
P_{4}\left(\frac{1}{\sqrt{n}}\right)\\
P_{5}\left(\frac{1}{\sqrt{n}}\right)\\
P_{6}\left(\frac{1}{\sqrt{n}}\right)\\
P_{7}\left(\frac{1}{\sqrt{n}}\right)\\
1\end{array}
\right)^{T},\]
and sequence $\alpha^{*}_{n}=A\left(\frac{1}{\sqrt{n}}\right)$, where $A$ and each $P_{i}$ is a polynomial (of degree at most 4) with unknown coefficients.

The equation
\begin{equation}
\label{eq:R_generic_equation}
\left|\Rvec^{*}_{n+1}\Mvec_{n}- \alpha^{*}_{n} \Rvec^{*}_{n}\right|=O\left(n^{-\frac{5}{2}}\right),
\end{equation}
then defines a system of simultaneous polynomial equations satisfied by the coefficients of these polynomials. Solving these equations in SageMath (see accompanying SageMath notebook) yields five solutions, of which we only need the following three:
\[\Rvec^{(0)}_{n}:=\left(\begin{array}{c}
0\\
0\\
0\\
0\\
\frac{18}{n^{2}}\\
\frac{3}{n}-\frac{3 \, \sqrt{3}}{n^{\frac{3}{2}}} - \frac{143}{4 \, n^{2}}\\
1-\frac{\sqrt{3}}{\sqrt{n}}  - \frac{23}{12 \, n}+ \frac{5 \, \sqrt{3}}{6 \, n^{\frac{3}{2}}} + \frac{2941}{144 \, n^{2}}\\
 - 2+\frac{\sqrt{3}}{\sqrt{n}} - \frac{13}{12 \, n} + \frac{\sqrt{3}}{2 \, n^{\frac{3}{2}}} + \frac{275}{12 \, n^{2}}\\
1
\end{array}
\right)^{T} ~~~~~~ \Rvec^{(1)}_{n}:=\left(\begin{array}{c}
0\\
0\\
0\\
0\\
\frac{18}{n^{2}}\\
\frac{3}{n} - \frac{41}{n^{2}}\\
1-\frac{11}{3 \, n} + \frac{148}{9 \, n^{2}}\\
-2-\frac{7}{3 \, n} + \frac{65}{3 \, n^{2}}\\
1
\end{array}
\right)^{T}\]

\[\Rvec^{(2)}_{n}:=\left(\begin{array}{c}
0\\
0\\
0\\
0\\
\frac{18}{n^{2}}\\
\frac{3}{n}+\frac{3 \, \sqrt{3}}{n^{\frac{3}{2}}} - \frac{143}{4 \, n^{2}}\\
1+\frac{\sqrt{3}}{\sqrt{n}}  - \frac{23}{12 \, n}-\frac{5 \, \sqrt{3}}{6 \, n^{\frac{3}{2}}} + \frac{2941}{144 \, n^{2}}\\
 - 2-\frac{\sqrt{3}}{\sqrt{n}} - \frac{13}{12 \, n} - \frac{\sqrt{3}}{2 \, n^{\frac{3}{2}}} + \frac{275}{12 \, n^{2}}\\
1
\end{array}
\right)^{T},\]
along with 
\begin{align*}
\alpha^{(0)}_{n}&=1+\frac{\sqrt{3}}{\sqrt{n}}+\frac{23}{12n},\\
\alpha^{(1)}_{n}&=1+\frac{2}{3n},\\
\alpha^{(2)}_{n}&=1-\frac{\sqrt{3}}{\sqrt{n}}+\frac{23}{12n}.
\end{align*}

Finally, we need to choose appropriate $\lambda_n^{(1)},\lambda_n^{(2)}$ such that $y_n = x_n^{(0)} + \lambda_n^{(1)} x_n^{(1)} +\lambda_n^{(2)} x_n^{(2)}$ has asymptotic behaviour closely related to that of $\widehat{t}_n$. It is enough to require that $(\Rvec_n^{(0)} +\lambda_n^{(1)}\Rvec_n^{(1)} +\lambda_n^{(2)}\Rvec_n^{(2)})$ has non-negative coefficients, as then $y_n = (\Rvec_n^{(0)} +\lambda_n^{(1)}\Rvec_n^{(1)} +\lambda_n^{(2)}\Rvec_n^{(2)})\Tvec_n$ is a positive linear combination of consecutive terms $\widehat{t}_{n+k}$. It is easy to check that $(\lambda_n^{(1)},\lambda_n^{(2)}) = (-2,1)$ for all $n \geqslant 1$ satisfies this property. 

The following lemma summarises the properties of the sequences $(\Rvec_n^{(j)})_n$ and $(\alpha_n^{(j)})_n$ that we will use in the rest of the proof.

\begin{lemma}\label{lem:RM_computations}
For each $j \in \{0,1,2\}$, we have 
\begin{equation}
\label{eq:RM_equation}
\left|\Rvec^{(j)}_{n+1}\Mvec_{n}- \alpha^{(j)}_{n} \Rvec^{(j)}_{n}\right|=O\left(n^{-\frac{5}{2}}\right),
\end{equation}
moreover the following linear combination has non-negative coefficients:
\begin{equation}
\label{eq:R_positive_linear_combination}
\Rvec^{(0)}_{n}-2\Rvec^{(1)}_{n}+\Rvec^{(2)}_{n}=\left(\begin{array}{c}
0\\
0\\
0\\
0\\
0\\
\frac{21}{2 \, n^{2}}\\
\frac{7}{2 \, n} + \frac{191}{24 \, n^{2}}\\
\frac{5}{2 \, n} + \frac{5}{2 \, n^{2}}\\
0
\end{array}
\right)^{T}
\end{equation}
\end{lemma}
\begin{proof}
The equation \eqref{eq:RM_equation} follows directly from the way we constructed the matrices. Then, \eqref{eq:R_positive_linear_combination} follows from directly computing the matrices (see accompanying SageMath file ``Triangular percolation worksheet'')
\end{proof}
\begin{rem}
 We note that we could have used a linear combination of all five solutions to \eqref{eq:R_generic_equation} instead of just three to produce, for example, the vector $(0~0~\cdots~0~1)$ on the right of \eqref{eq:R_positive_linear_combination}. However, doing so, the coefficients $\lambda_n^{(j)}$ would depend on $n$, and in particular their sum may not be bounded. This would complicate the rest of the proof, and thus we prefer to use the linear combination considered above. 
\end{rem}

\subsection{A dichotomy for \texorpdfstring{$(\widehat{t}_n)_{n \geqslant 0}$}{tnhat}}
\label{subsec:dichotomy}

We prove here that we can apply Theorem \ref{thm:dichotomy_general} to the sequences $(x_n^{(j)})_{n \geqslant 1}$ and $(\alpha_n^{(j)})_{n \geqslant 1}$ defined above, and therefore get the following dichotomy result for $(\widehat{t}_n)_{n \geqslant 0}$. 

\begin{prop}
\label{prop:dichotomy applied to tn}
Either $\widehat{t}_n$ grows at most polynomially, or there exists a constant $c>0$ such that
\begin{align*}
    \widehat{t}_n \underset{n \rightarrow\infty}{\sim} c e^{\sqrt{12 n}} n^{17/12}
\end{align*}
\end{prop}

Later, in Lemma \ref{lem:l_growth}, we will show that in fact the polynomial case does not occur, completing the proof of Theorem \ref{thm:asymptotic_behaviour}.

\begin{proof}[Proof of Proposition \ref{prop:dichotomy applied to tn}]
We start by applying Theorem \ref{thm:dichotomy_general} using $\varepsilon = 1/8$, $(\alpha_n^{(j)})_{n \geqslant 1}$ and $(x_n^{(j)})_{n \geqslant 1}$ as defined above, and $y_{n}=x_{n}^{(0)}-2x_{1}^{(0)}+x_{n}^{(2)}$. To do this we need to prove that the sequences $(x_n^{(j)})_{n \geqslant 1}$ for $j \in \{0, 1, 2\}$ and $(y_n)_{n \geqslant 1}$ satisfy \[n^{-5/4}y_{n}>|x_{n+1}^{(j)}-\alpha_{n}^{(j)}x_{n}^{(j)}|,\]
for $n$ sufficiently large. 

Using \eqref{eq:RM_equation}, we can choose $N$ sufficiently large such that for $j\in\{0,1,2\}$ and all $n>N$ we have 
\begin{equation}
\label{eq:RM_upper_bound}
    \left|\Rvec^{(j)}_{n+1}\Mvec_{n}- \alpha_{n}^{(j)} \Rvec^{(j)}_{n}\right|<\frac{1}{4}n^{-9/4}.
\end{equation}
We may assume additionally that \begin{equation}\label{eq:mhatsum_upper_bound}
    |\widehat{m}_{1}(n)|+\cdots+|\widehat{m}_{9}(n)|<8,
\end{equation}
for $n>N$, since this sum converges to $7$ as $n\to\infty$.

Then by \eqref{eq:R_positive_linear_combination},
\begin{align*}
y_{n}&=(\Rvec^{(0)}_{n}-2\Rvec^{(1)}_{n}+\Rvec^{(2)}_{n})\Tvec_{n}\\ 
&=\frac{21}{2n^{2}}\widehat{t}_{n+5} + \left(\!\frac{7}{2n} + \frac{191}{24n^{2}}\!\right)\widehat{t}_{n+6} + \left(\!\frac{5}{2n} + \frac{5}{2n^{2}}\!\right)\widehat{t}_{n+7}\\ &>\frac{7}{2n}\widehat{t}_{n+6}+\frac{5}{2n}\widehat{t}_{n+7}.
\end{align*}

Now observe that, since $t_n \geqslant nt_{n-1}$, we have $\widehat{t}_n > \widehat{t}_{n-1}$ (see Section~\ref{sec:recurrences}). Therefore we have $\widehat{t}_{n}< \cdots < \widehat{t}_{n+7}$, and then
from \eqref{eq:mhatsum_upper_bound} we have $\widehat{t}_{n+8}< 8\widehat{t}_{n+7}$. Together, this implies
\begin{equation}
\label{eq:y_lower_bound}
y_{n}>\frac{7}{2n}\widehat{t}_{n+6}+\frac{5}{2n}\widehat{t}_{n+7}>\frac{1}{2n}\left(\widehat{t}_{n}+ \cdots+ \widehat{t}_{n+6}\right)+\frac{5}{18n} \left(\widehat{t}_{n+7}+ \widehat{t}_{n+8}\right) > \frac{5}{18n}  |\Tvec_{n}|.
\end{equation}

Using \eqref{eq:y_lower_bound}, \eqref{eq:RM_upper_bound}, \eqref{eq:tlhat_matrix_recurrence}, and the fact that $\Rvec_{n+1}^{(j)} \Fvec_{n+1} = \widehat{f}_{n+9}$ by definition, we can write for $n$ large enough:
\begin{equation}
\label{eq:jpp}
n^{-5/4}y_{n}> \frac{5}{18}n^{-9/4}|\Tvec_n|\geqslant \frac{20}{18} \left|\left(\Rvec^{(j)}_{n+1}\Mvec_{n}- \alpha_{n}^{(j)} \Rvec^{(j)}_{n}\right)\Tvec_{n}\right| = \frac{10}{9}\left|x_{n+1}^{(j)}-\alpha_{n}^{(j)}x_{n}^{(j)}-\widehat{f}_{n+9}\right|.
\end{equation}

Moreover, using Lemma \ref{lem_boundFT} and \eqref{eq:y_lower_bound}, we have $\widehat{f}_{n+9}=O(\widehat{t}_{n+6}/n^3)=O(y_{n}/n^2)$. Hence, for $n$ large enough, $\frac{1}{10} n^{-5/4} y_n \geqslant |\widehat{f}_{n+9}|$. Hence, using \eqref{eq:jpp}, we have for $n$ large enough:
\begin{align*}
n^{-5/4} y_n &= \frac{9}{10}n^{-5/4} y_n + \frac{1}{10} n^{-5/4} y_n\\
& \geqslant |x_{n+1}^{(j)}-\alpha_{n}^{(j)}x_{n}^{(j)}-\widehat{f}_{n+9}| + |\widehat{f}_{n+9}|\\
&\geqslant |x_{n+1}^{(j)}-\alpha_{n}^{(j)}x_{n}^{(j)}|,
\end{align*}
as required. Finally this implies that we can apply Theorem \ref{thm:dichotomy_general}. 

\hfill

We now know that either $(y_n)_{n\geqslant 1}$ grows at most polynomially or $y_{n}\sim ce^{H(n)}n^\gamma$ for some constant $c>0$ where, by definition of $\alpha_n^{(0)}$, we have $H(n)=2\sqrt{3n}$ and $\gamma=23/12-3/2 = 5/12$. So $(ny_n)_{n\geqslant 1}$ satisfies the dichotomy we announced for $(\widehat{t}_n)_{n\geqslant 1}$. All that remains is to prove that $(\widehat{t}_n)_{n\geqslant 1}$ and $(ny_n)_{n\geqslant 1}$ satisfy the same dichotomy.

Recall that by definition
\begin{equation}
\label{eq:jsplol}
y_{n}=\frac{21}{2n^{2}}\widehat{t}_{n+5}+\left(\frac{7}{2n}+\frac{191}{24n^{2}}\right)\widehat{t}_{n+6}+\left(\frac{5}{2n}+\frac{5}{2n^{2}}\right)\widehat{t}_{n+7}.
\end{equation}
Hence if $y_n$ grows at most polynomially, so does $\widehat{t}_n$, as both sequences are positive. 
Assume now that $y_n \sim c e^{\sqrt{12 n}} n^{5/12}$ for some $c>0$, as $n \rightarrow \infty$. 
Then, since $\widehat{t}_{n}$ is an increasing sequence, we get from \eqref{eq:jsplol} that
\[\frac{(n-5)^2}{6(n-5)+503/24}y_{n-5}\geqslant\widehat{t}_{n}\geqslant \frac{(n-7)^2}{6(n-7)+503/24}y_{n-7}.\]
The upper and lower bounds are both asymptotically equivalent to $\frac{c}{6} e^{\sqrt{12 n}}n^{17/12}$, so $\widehat{t}_{n}$ also has this behaviour, as claimed.
\end{proof}

\subsection{Superpolynomial growth of \texorpdfstring{$(\widehat{t}_n)_{n \geqslant 0}$}{tnhat}}
\label{subsec:ruleout}

All that remains is to prove that the sequence $(\widehat{t}_n)_{n \geqslant 0}$ grows superpolynomially fast. To do this, we introduce a sequence of lower bounds $\widehat{\ell}_n$ for the numbers $\widehat{t}_n$ with the same asymptotic behaviour. 

More precisely, we consider the sequence $(\ell_n)_{n\geqslant0}$ where $\ell_{n}$ is the number of bases of size $n$ whose $3$-core has size $N$, for some $N$ to be specified later. 
It is clear from this definition that $t_{n}\geqslant \ell_{n}$. 
Additionally $\ell_{0}=\cdots=\ell_{N-1}=0$ and $\ell_{N}=f_{3,N}>0$. Moreover,  from \eqref{eq:inclus-exclus_k_m}, for $n\geqslant N$, $\ell_{n}$ is determined by the recurrence
\begin{equation}
\label{eq:l_recurrence}
\ell_{n+9}=m_{1}(n)\ell_{n+8}+m_{2}(n)\ell_{n+7}+m_{3}(n)\ell_{n+6}+\cdots+m_{8}(n)\ell_{n+1}+m_{9}(n)\ell_{n},
\end{equation}
with the same $m_i(n)$ as in \eqref{eq:t_and_f_recurrence}.
Observe that this is the same recurrence as for $t_n$ except that the $f_{n+9}$ is missing. This follows from the fact that the difference between \eqref{eq:inclus-exclus_k_m} and \eqref{eq:inclus-exclus_Fk} is precisely the term $F(z)$. 

Analogously to the notation used for $(t_n)_n$, we define   $\widehat{\ell_n} = \frac{\ell_n}{(n-1)!}$ and 
%
\[\Lvec_{n}:=\left(\begin{array}{c}
\widehat{\ell}_{n}\\
\widehat{\ell}_{n+1}\\
\vdots\\
\widehat{\ell}_{n+6}\\
\widehat{\ell}_{n+7}\\
\widehat{\ell}_{n+8}
\end{array}
\right).\]

Then \eqref{eq:l_recurrence} can be written as 
\[\Mvec_{n}\Lvec_{n}=\Lvec_{n+1}.\]

We can now apply Theorem~\ref{thm:dichotomy_general} to the sequence $x_{n}^{(j)}=\Rvec^{(j)}_{n}\Lvec_{n}$, with the same $(\Rvec_n^{(j)})_{n \geqslant 1}$ as before, to obtain the asymptotic behaviour of $\widehat{\ell}_n$.

\begin{lemma}
\label{lem:l_growth}
There is a constant $N$ such that the sequence $\ell_n$ of bases of size $n$ with $3$-core of size $N$ satisfies
\[\ell_{n}\sim c n!e^{\sqrt{12 n}}n^{5/12},\]
for some constant $c>0$.
\end{lemma}
\begin{proof}
Let $N$ be large enough so that Theorem~\ref{thm:dichotomy_general} applies for $\varepsilon =1/8$ and the same $\alpha_n^{(j)}$ as in Proposition \ref{prop:dichotomy applied to tn}, and such that \eqref{eq:RM_upper_bound} and \eqref{eq:mhatsum_upper_bound} hold for $N-8$.

For each $j$, write $x_{n}^{(j)}=\Rvec^{(j)}_{n}\Lvec_{n}$, and $y_{n}=x_{n}^{(0)}-2x_{n}^{(1)}+x_{n}^{(2)}$. With the exact same arguments as in Proposition \ref{prop:dichotomy applied to tn}, it holds that, for the same $\alpha_n^{(j)}$, for each $j \in \{0,1,2\}$
\[n^{-5/4}y_{n}>|x_{n+1}^{(j)}-\alpha_{n}^{(j)}x_{n}^{(j)}|\]
for $n$ large enough, allowing us to use Theorem \ref{thm:dichotomy_general}.

Now observe that by definition, for each $j \in \{0,1,2\}$ we have
\[x_{N-8}^{(j)}=\Rvec^{(j)}_{N-8}\Lvec_{N-8}=\widehat{\ell}_{N},\]
so
\[x_{N-8}^{(0)}=x_{N-8}^{(1)}=x_{N-8}^{(2)}.\]
Hence, the first case of Theorem \ref{thm:dichotomy_general} holds, which implies that
\[y_{n}\sim x_{n}^{(0)}\sim c e^{\sqrt{12 n}}n^{5/12},\]
for some constant $c>0$.

Finally, using the definition of $y_n$ in the same way as in Proposition~\ref{prop:dichotomy applied to tn}, we obtain
\[\ell_{n}\sim \frac{c}{6}n! e^{\sqrt{12 n}}n^{5/12},\]
as claimed.

\end{proof}

We finally have all the ingredients to prove Theorem \ref{thm:asymptotic_behaviour}.

\begin{proof}[Proof of Theorem~\ref{thm:asymptotic_behaviour}]
From Lemma \ref{lem:l_growth}, the sequence $(\widehat{\ell}_n)_{n\geqslant0}$ grows superpolynomially. Since $\widehat{t}_n\geqslant\widehat{\ell}_n$ for all $n$, Proposition \ref{prop:dichotomy applied to tn} implies that
\[\widehat{t}_n \sim c e^{\sqrt{12 n}}n^{17/12},\]
for some constant $c>0$. Finally, this implies
\[t_{n} \sim c n! e^{\sqrt{12 n}}n^{5/12},\]
as claimed.
\end{proof}

\section{Discussions}

We list here several projects and open questions, which we think are promising.

\subsection{Asymptotic structure of a triangle basis chosen uniformly at random}
Asymptotic enumeration of discrete objects is often used to understand the limiting behaviour of these objects as their size approaches infinity. We therefore pose the natural question of understanding the structure of large random triangle bases. 

As a first step, by Theorem \ref{thm:asymptotic_behaviour}, the proportion of bases of size $n$ that have only one point on their bottom line is 
\[\frac{nt_{n-1}}{t_n} \underset{n\to\infty}{\to} 1.\]
So with high probability, a large uniform random basis has only one point on its bottom line, which is distributed uniformly at random on that line. In addition, by rotational symmetry, with high probability it has exactly one (uniformly distributed) point on each of its sides. 

Hence, one would expect a large uniform basis to have a higher density of points closer to the middle of the triangle, and a lower density close to the sides. More generally, it would be interesting to compute the expected number of points of a uniformly random large basis that fall into a given region of the triangle.

Another interesting question to look at is the expected size of the $k$-core of a large uniform random basis. We expect it to be small for $k \geqslant 2$, possibly of constant ordre.


\smallskip

A similar question of interest concerns the asymptotic structure of a uniformly random element in the class $Av_n((12,12), (231, 312))$ of pattern avoiding $3$-permutations. Several models of permutations taken uniformly at random in a pattern avoiding class have been shown to converge to some limit object called a \emph{permuton}, which is a probability measure on the unit square, see for instance \cite{permuton_bassino}.
One can therefore also wonder if a uniform $3$-permutation of size $n$ in the class $Av((12,12), (231, 312))$ converges to a $3$-dimensional permuton, as defined in \cite{BoLi25}. Again our asymptotic results should be an important stepping stone to answering this question.

\subsection{Related models}
There is a model of bootstrap percolation on the square lattice, known as \emph{Froböse percolation}, analogous to ours (see \cite{HT24, SS25}). In that model, one can add a point $x$ to a configuration $\eta$ if there exists a $2$ by $2$ square $\{x,y,z,t\}$ such that $y,z,t \in \eta$. The same way as in our model, one can define bases of Froböse percolation on finite rectangles as minimal configurations that percolate to a given rectangle. In a private communication, Hartarsky conjectured that the numbers of bases of Froböse percolation on a finite-size square also have a stretched exponential factor in their asymptotics, which is supported by simulations. However, this model is naturally bivariate, with one variable accounting for the length, and the other one for the width of the rectangle, which makes the asymptotic study more intricate.

In whole generality, there are also many more models of pattern avoiding permutations or higher dimensional permutations that one could study. A long-term goal would be to understand which such models have a stretched exponential behaviour and which do not. 

\subsection{Generality of the method}
We hope that our method, which seems fairly robust, and in particular the technical Theorem~\ref{thm:dichotomy_general}, may be used as a blackbox to rigorously show that stretched exponential asymptotic behaviour appears in combinatorial sequences. A main obstacle that remains in many cases where a stretched exponential is conjectured is the lack of a recursive decomposition and the associated equations, in particular in the field of pattern avoiding permutations. 

We also hope that our method can be extended to provide a nice constructive proof that the coefficients of any D-finite generating function whose associated equation has an irregular dominant singularity has a stretched exponential asymptotic behaviour.

Another potential extension of the method is to quasi-linear \emph{bivariate} recurrences. A method for deducing asymptotics for bivariate \emph{linear} recurrences has recently been developed \cite{EPFW21, EFLW25}, which also yields conjectural results for some bivariate non-linear recurrences (see \cite[Conjecture 25]{EFLW25}). We hope that some of the ideas we used to replace a non-linear recurrence with a quasi-linear recurrence will also apply to these bivariate recurrences. In particular, we will need to extend our method in this way in order to address asymptotics of the number of bases of Froböse percolation. 

\section{Acknowledgements}

The authors would like to thank Alin Bostan, Mireille Bousquet-Mélou, Michèle Loday-Richaud and Marc Mezzarobba for fruitful discussions about singularity analysis of D-finite functions, as well as Ivailo Hartarsky for the initial motivation. 

\hfill

Andrew Elvey Price was partially supported by ANR HighGG (ANR-24-CE40-2078-01).

Paul Thévenin was partially supported by the project ANR-24-CPJ1-0032-01, and the
project "Rawabranch" number ANR-23-CE40-0008.

\bibliographystyle{alpha}
\bibliography{biblioBases}

@book {Wasow,
    AUTHOR = {Wasow, Wolfgang},
     TITLE = {Asymptotic expansions for ordinary differential equations},
      NOTE = {Reprint of the 1965 edition},
 PUBLISHER = {Robert E. Krieger Publishing Co., Huntington, NY},
      YEAR = {1976},
     PAGES = {ix+374},
      ISBN = {0-88275-173-5},
   MRCLASS = {34EXX},
  MRNUMBER = {460820},
}

@incollection {SS23,
    AUTHOR = {Salo, Ville and Schabanel, Juliette},
     TITLE = {Triangle solitaire},
 BOOKTITLE = {Cellular automata and discrete complex systems},
    VOLUME = {14152},
     PAGES = {123--136},
 PUBLISHER = {Springer, Cham},
      YEAR = {2023},
       DOI = {10.1007/978-3-031-42250-8\_9},
}

@article {CG15,
    AUTHOR = {Conway, Andrew R. and Guttmann, Anthony J.},
     TITLE = {On 1324-avoiding permutations},
   JOURNAL = {Adv. in Appl. Math.},
  FJOURNAL = {Advances in Applied Mathematics},
    VOLUME = {64},
      YEAR = {2015},
     PAGES = {50--69},
      ISSN = {0196-8858,1090-2074},
   MRCLASS = {05A05 (05A15 05A16)},
  MRNUMBER = {3300327},
MRREVIEWER = {Antonio\ Bernini},
       DOI = {10.1016/j.aam.2014.12.004},
       URL = {https://doi.org/10.1016/j.aam.2014.12.004},
}

@article {CGZJ18,
    AUTHOR = {Conway, Andrew R. and Guttmann, Anthony J. and {Zinn-Justin},
              Paul},
     TITLE = {1324-avoiding permutations revisited},
   JOURNAL = {Adv. in Appl. Math.},
  FJOURNAL = {Advances in Applied Mathematics},
    VOLUME = {96},
      YEAR = {2018},
     PAGES = {312--333},
      ISSN = {0196-8858,1090-2074},
   MRCLASS = {05A05 (05A15 05A16)},
  MRNUMBER = {3767512},
MRREVIEWER = {David\ Bevan},
       DOI = {10.1016/j.aam.2018.01.002},
       URL = {https://doi.org/10.1016/j.aam.2018.01.002},
}

@article {BW21,
    AUTHOR = {Banderier, Cyril and Wallner, Michael},
     TITLE = {Young tableaux with periodic walls: counting with the density
              method},
   JOURNAL = {S\'em. Lothar. Combin.},
  FJOURNAL = {S\'eminaire Lotharingien de Combinatoire},
    VOLUME = {85B},
      YEAR = {2021},
     PAGES = {Art. 47, 12},
      ISSN = {1286-4889},
   MRCLASS = {05A15 (05A19)},
  MRNUMBER = {4311928},
}

@article {FYZ21,
    AUTHOR = {Fuchs, Michael and Yu, Guan-Ru and Zhang, Louxin},
     TITLE = {On the asymptotic growth of the number of tree-child networks},
   JOURNAL = {European J. Combin.},
  FJOURNAL = {European Journal of Combinatorics},
    VOLUME = {93},
      YEAR = {2021},
     PAGES = {Paper No. 103278, 20},
      ISSN = {0195-6698,1095-9971},
   MRCLASS = {05C30 (05C05 92D15)},
  MRNUMBER = {4186619},
MRREVIEWER = {Shouliu\ Wei},
       DOI = {10.1016/j.ejc.2020.103278},
       URL = {https://doi.org/10.1016/j.ejc.2020.103278},
}

@article {EPG19,
    AUTHOR = {{Elvey Price}, Andrew and Guttmann, Anthony J.},
     TITLE = {Numerical studies of {T}hompson's group {$F$} and related
              groups},
   JOURNAL = {Internat. J. Algebra Comput.},
  FJOURNAL = {International Journal of Algebra and Computation},
    VOLUME = {29},
      YEAR = {2019},
    NUMBER = {2},
     PAGES = {179--243},
      ISSN = {0218-1967,1793-6500},
   MRCLASS = {20F65 (05A05 05A15 20F05)},
  MRNUMBER = {3934784},
MRREVIEWER = {J\"org\ Lehnert},
       DOI = {10.1142/S0218196718500686},
       URL = {https://doi.org/10.1142/S0218196718500686},
}

@book {FS09,
    AUTHOR = {Flajolet, Philippe and Sedgewick, Robert},
     TITLE = {Analytic combinatorics},
 PUBLISHER = {Cambridge University Press, Cambridge},
      YEAR = {2009},
     PAGES = {xiv+810},
      ISBN = {978-0-521-89806-5},
   MRCLASS = {05-02 (05A15 05A16 60C05 60E10 82-01)},
  MRNUMBER = {2483235},
       DOI = {10.1017/CBO9780511801655},
       URL = {https://doi.org/10.1017/CBO9780511801655},
}

@article {SS25,
    AUTHOR = {Salo, Ville and Schabanel, Juliette},
     TITLE = {Solitaire of independence},
   JOURNAL = {Nat. Comput.},
  FJOURNAL = {Natural Computing. An International Journal},
    VOLUME = {24},
      YEAR = {2025},
    NUMBER = {3},
     PAGES = {431--467},
      ISSN = {1567-7818,1572-9796},
   MRCLASS = {99-03},
  MRNUMBER = {4957084},
       DOI = {10.1007/s11047-025-10010-3},
       URL = {https://doi.org/10.1007/s11047-025-10010-3},
}

@article{HT24,
  title={Bootstrap percolation is local},
  author={Hartarsky, Ivailo and Teixeira, Augusto},
  journal={ArXiv preprint arXiv:2404.07903},
  year={2024}
}

@article{EFLW25,
  title={Bivariate asymptotics via random walks: application to large genus maps},
  author={{Elvey Price}, Andrew and Fang, Wenjie and Louf, Baptiste and Wallner, Michael},
  journal={arXiv preprint arXiv:2506.06924},
  year={2025}
}

@article {Rev03,
    AUTHOR = {Revelle, David},
     TITLE = {Heat kernel asymptotics on the lamplighter group},
   JOURNAL = {Electron. Comm. Probab.},
  FJOURNAL = {Electronic Communications in Probability},
    VOLUME = {8},
      YEAR = {2003},
     PAGES = {142--154},
      ISSN = {1083-589X},
   MRCLASS = {60G50 (60J10)},
  MRNUMBER = {2042753},
MRREVIEWER = {Donald\ I.\ Cartwright},
       DOI = {10.1214/ECP.v8-1092},
       URL = {https://doi.org/10.1214/ECP.v8-1092},
}

@article {EPFW21,
    AUTHOR = {{Elvey Price}, Andrew and Fang, Wenjie and Wallner, Michael},
     TITLE = {Compacted binary trees admit a stretched exponential},
   JOURNAL = {J. Combin. Theory Ser. A},
  FJOURNAL = {Journal of Combinatorial Theory. Series A},
    VOLUME = {177},
      YEAR = {2021},
     PAGES = {Paper No. 105306, 40},
      ISSN = {0097-3165,1096-0899},
   MRCLASS = {05A15 (05A16 68Q45)},
  MRNUMBER = {4143528},
MRREVIEWER = {David\ Bevan},
       DOI = {10.1016/j.jcta.2020.105306},
       URL = {https://doi.org/10.1016/j.jcta.2020.105306},
}

@incollection {EPFW20,
    AUTHOR = {{Elvey Price}, Andrew and Fang, Wenjie and Wallner, Michael},
     TITLE = {Asymptotics of minimal deterministic finite automata
              recognizing a finite binary language},
 BOOKTITLE = {31st {I}nternational {C}onference on {P}robabilistic,
              {C}ombinatorial and {A}symptotic {M}ethods for the {A}nalysis
              of {A}lgorithms},
    SERIES = {LIPIcs. Leibniz Int. Proc. Inform.},
    VOLUME = {159},
     PAGES = {Art. No. 11, 13},
 PUBLISHER = {Schloss Dagstuhl. Leibniz-Zent. Inform., Wadern},
      YEAR = {2020},
      ISBN = {978-3-95977-147-4},
   MRCLASS = {68Q45},
  MRNUMBER = {4116404},
}

@article{BijPermut,
      title={3{D} permutations and triangle solitaire}, 
      author={Juliette Schabanel},
      year={2025},
      journal={ArXiv preprint arXiv:2502.14657},
      url={https://arxiv.org/abs/2502.14657}, 
}

@article {Fro89,
    AUTHOR = {Frob\"ose, K.},
     TITLE = {Finite-size effects in a cellular automaton for diffusion},
   JOURNAL = {J. Statist. Phys.},
  FJOURNAL = {Journal of Statistical Physics},
    VOLUME = {55},
      YEAR = {1989},
    NUMBER = {5-6},
     PAGES = {1285--1292},
      ISSN = {0022-4715,1572-9613},
   MRCLASS = {82-08 (82A43)},
  MRNUMBER = {1002492},
       DOI = {10.1007/BF01041088},
       URL = {https://doi.org/10.1007/BF01041088},
}

@book{Stanley2,
  title = {Enumerative {{Combinatorics}}},
  author = {Stanley, Richard P.},
  year = 2011,
  series = {Cambridge {{Studies}} in {{Advanced Mathematics}}},
  edition = {2},
  volume = {2},
  publisher = {Cambridge University Press},
  doi = {10.1017/CBO9780511609589}
}

@article{Bonichon_d-permutations,
  title={Baxter d-Permutations and Other Pattern-Avoiding Classes},
  author={Bonichon, Nicolas and Morel, Pierre-Jean},
  journal={Journal of Integer Sequences},
  volume={25},
  number={8},
  pages={22--8},
  year={2022}
}

@article{Gunby_d-permutations,
  title={Asymptotics of pattern avoidance in the Klazar set partition and permutation-tuple settings},
  author={Gunby, Benjamin and P{\'a}lv{\"o}lgyi, D{\"o}m{\"o}t{\"o}r},
  journal={European Journal of Combinatorics},
  volume={82},
  pages={102992},
  year={2019},
  publisher={Elsevier}
}

@Article{permuton_bassino,
 Author = {Bassino, Fr{\'e}d{\'e}rique and Bouvel, Mathilde and F{\'e}ray, Valentin and Gerin, Lucas and Maazoun, Micka{\"e}l and Pierrot, Adeline},
 Title = {Universal limits of substitution-closed permutation classes},
 FJournal = {Journal of the European Mathematical Society (JEMS)},
 Journal = {J. Eur. Math. Soc.},
 Volume = {22},
 Number = {11},
 Year = {2020},
 Pages = {3565--3639},
 Language = {English},
 DOI = {10.4171/JEMS/993},
 zbMATH = {7286839},
 Zbl = {1469.60039}
}

@article{BoLi25,
      title={High-dimensional permutons: theory and applications}, 
      author={Jacopo Borga and Andrew Lin},
      year={2025},
      journal={ArXiv preprint arXiv:2412.19730},
      url={https://arxiv.org/abs/2412.19730},
}

@article{ALW1995,
  title={Priority queues and multisets},
  author={Atkinson, Mike D. and Linton, Stephen A. and Walker, Louise A.},
  journal={the electronic journal of combinatorics},
  pages={R24--R24},
  year={1995}
}

@article{aldred2005permuting,
  title={Permuting machines and priority queues},
  author={Aldred, Robert E.L. and Atkinson, Mike D. and van Ditmarsch, Hans P. and Handley, Chris C. and Holton, Derek A. and McCaughan, Dennis J.},
  journal={Theoretical computer science},
  volume={349},
  number={3},
  pages={309--317},
  year={2005},
  publisher={Elsevier}
}

\end{document}